%% file: non-Archimedean_Hitchin_map.tex
\documentclass[11pt]{amsart}
\usepackage{amsmath, amssymb, amscd, mathrsfs, slashed, enumerate, url}
\usepackage{color}
\usepackage{extsizes}
\usepackage{xcolor}
\usepackage[all,cmtip]{xy}
\usepackage{multicol}
\usepackage{enumitem} 
\usepackage{indentfirst}
\usepackage{tikz}
\usetikzlibrary{calc}
\usepackage{latexsym}
\usepackage{bm}
\usepackage{graphicx}
\usepackage{subfigure,esint}
\usepackage{float}
\usepackage{verbatim}
\usepackage{comment}
\usepackage[top=1in, bottom=1in, left=1.25in, right=1.25in]{geometry}
\usepackage{epsfig,dsfont,amsthm,amsfonts,amsbsy}
\usepackage[colorlinks]{hyperref}
\usepackage[toc,page]{appendix}
\usepackage{tikz-cd}
\usepackage{amsthm,thmtools,xcolor}
\newtheorem{theorem}{Theorem}[section]

\newtheorem{corollary}[theorem]{Corollary}

\newtheorem{definition}[theorem]{Definition}
\newtheorem{question}[theorem]{Question}

\newtheorem{lemma}[theorem]{Lemma}

\newtheorem{proposition}[theorem]{Proposition}
\newtheorem*{remark}{Remark}

\input{convention.tex}

\usepackage{calligra}
\DeclareMathAlphabet{\mathcalligra}{T1}{calligra}{m}{n}

\declaretheoremstyle[
headfont=\color{blue}\normalfont\bfseries,
bodyfont=\color{blue}\normalfont\itshape,
]{colored}

\usepackage[utf8]{inputenc}
\usepackage[english]{babel}
\usepackage{fancyhdr}
\usepackage{accents}

\newcommand{\Addresses}{{% additional braces for segregating \footnotesize
  \bigskip
  \footnotesize

  Jiahuang Chen: \textsc{Institute of Mathematics, Academy of Mathematics and Systems Science, Chinese
 Academy of Sciences, Beijing, 100190, P.\,R.\,China}\par\nopagebreak
  \textit{E-mail address}: \texttt{chenjiahuang@amss.ac.cn}

  \medskip

  Siqi He: \textsc{Morningside Center of Mathematics,
	Chinese Academy of Sciences, 
	Beijing, 100190, P.\,R.\,China}\par\nopagebreak
  \textit{E-mail address}: \texttt{sqhe@amss.ac.cn}

}}

\begin{document}
	\title[On the Non-Archimedean Hitchin Map for $\SL_2(F)$]{On the Non-Archimedean Hitchin Map for $\SL_2(F)$}

\author{Jiahuang Chen
        \and 
        Siqi He}

\maketitle

\begin{abstract}
Let $F$ be a non-Archimedean valued field, $\Sigma$ a closed Riemann surface of genus at least two, and $\Gamma$ its fundamental group. Building on the theory of equivariant harmonic maps into $\mathbb{R}$-trees, we study the non-Archimedean Hitchin map from the $\SL_2(F)$-character variety $\mathcal{X}_F(\Gamma)$, equipped with the non-Archimedean topology, to the space of holomorphic quadratic differentials on $\Sigma$. We prove that this map is continuous and that its image is contained in the space of Jenkins--Strebel differentials. Moreover, we establish a dynamical characterization of unbounded representations, showing that the induced action of $\Gamma$ on the Bruhat--Tits tree of $\SL_2(F)$ is never small.
\end{abstract}

\section{Introduction}
The study of representations of surface groups into Lie groups has deep connections with complex geometry, low-dimensional topology and dynamical systems. For a closed Riemann surface $\Sigma$ of genus $g\ge 2$,  classical non-abelian Hodge theory establishes a deep correspondence between the moduli space of reductive representations of $\Gamma$ into a complex reductive group $G$ and the moduli of Higgs bundles \cite{hitchin1987self}. A key feature of this correspondence is the Hitchin fibration, which maps the representation space to the vector space of holomorphic differentials on $\Sigma$. In rank two, this yields a holomorphic map from the character variety to the space of quadratic differentials $H^0(\Sigma,\mathcal{K}_\Sigma^{\otimes 2})$, endowing the moduli space with the structure of an algebraic integrable system \cite{hitchin1992lie}.

Given a reductive representation $\rho: \Gamma \to \mathrm{SL}_2(\mathbb{C})$, there exists a unique $\rho$-equivariant harmonic map $u: \widetilde{\Sigma} \to \mathbb{H}^3$, where $\mathbb{H}^3 = \SL_2(\mathbb{C})/\mathrm{SU}(2)$ is the associated symmetric space. Up to scaling, the image of $\rho$ under the classical Hitchin map coincides with the Hopf differential of $u$ on $\Sigma$.

In this paper, we study an analogous construction in the non-Archimedean setting, where the target group is $\mathrm{SL}_2(F)$ for a non-Archimedean valued field $F$ of characteristic zero. The Bruhat-Tits tree $T_F$ associated to $\mathrm{SL}_2(F)$ is a simply connected, one-dimensional simplicial complex on which $\SL_2(F)$ acts isometrically. It serves as the non-Archimedean analogue of the symmetric space $\mathbb{H}^3$. Motivated by this analogy, we consider the \emph{non-Archimedean Hitchin map}, which assigns to each representation $\rho: \Gamma \to \SL_2(F)$ the Hopf differential of a $\rho$-equivariant harmonic map $u: \tSigma \to T_F$, whenever such a map exists. The existence and uniqueness of such harmonic maps for reductive representations were established in \cite{GromovSchoen1992,korevaar1993sobolev,korevaari1997global,mese2002uniqueness}. See also \cite{wolf1995harmonic,WolfRealizing, Klingler2013,daskalopoulos1998character,daskalopoulos2000morganshalen}.

We first define the non-Archimedean topology on the representation variety $\mR_F(\Gamma)$ of the surface group $\Gamma$. The character variety $\mathcal{X}_F(\Gamma)$ is then defined to be the \emph{Hausorffization} of the quotient space $\mR_F(\Gamma)/\SL_2(F)$. Using the harmonic map theory, we define a map $\overline{\Phi}:\mR_{F}(\Gamma)\to H^0(\Sigma,\mathcal{K}_\Sigma^{\otimes 2})$. Then $\overline{\Phi}$ descends to the character variety $\mathcal{X}_F(\Gamma)$, yielding the \emph{non-Archimedean Hitchin map}:
\begin{align}
    \Phi:\mathcal{X}_{F}(\Gamma)\to H^0(\Sigma,\mathcal{K}_\Sigma^{\otimes 2}).
\end{align}
Our first main result concerns the continuity of this map:
\begin{theorem}\label{thm:continuity}
    Let $\mathcal{X}_F(\Gamma)$ be equipped with the non-Archimedean topology, and $H^0(\Sigma,\mathcal{K}_\Sigma^{\otimes 2})$ with the Euclidean topology. Then the Hitchin map $\Phi$ is continuous. 
\end{theorem}

The topology of non-Archimedean fields differs drastically from the Euclidean topology. For instance, the unit ball $\mathbb{Z}_p \subset \mathbb{Q}_p$ is both open and closed and homeomorphic to the Cantor set in $[0,1] \subset \mathbb{R}$. Given that $\mathcal{X}_{F}(\Gamma)$ is equipped with the non-Archimedean topology and that $\Phi$ is continuous, one naturally expects the image of $\Phi$ to form an special subset of $H^0(\Sigma, \mathcal{K}_\Sigma^{\otimes 2})$. 

On the other hand, Jenkins-Strebel differentials are distinguished by the property that almost all of their trajectories are closed. They play a central role in Teichm\"uller theory and the theory of measured foliations \cite{Masur75,hubbard1979quadratic,MasurWolf95,Markovic}. While $\mathcal{JS}(\Sigma)$ is a \emph{dense} subset of $H^0(\Sigma, \mathcal{K}_\Sigma^{\otimes 2})$ \cite{douadyhubbard75}, it has \emph{measure zero} due to the fact that each such a differential is uniquely determined by a finite admissible curve system and positive weights on the curves \cite{strebel84quadratic,hubbard1979quadratic,Wolf_JSdiffs,WolfRealizing,liu04existence}.

Our second main theorem gives a rather surprising characterization of the image $\mathrm{Im}\Phi$:
\begin{theorem}\label{thm:imageJS}
    The image of $\Phi$ is contained in the subspace $\mathcal{JS}(\Sigma)$ of Jenkins-Strebel differentials on $\Sigma$.
\end{theorem}
The proof relies on the fact that the Bruhat–Tits tree $T_F$ is simplicial, which implies that every trajectory of the Hopf differential is nowhere dense---a trick already employed in \cite{Wolf_JSdiffs}. This observation agrees with the equivalent characterization of Jenkins-Strebel differentials.

Moreover, we also consider some dynamical consequences for a representation $\rho$. Recall that a group action on a tree is called \emph{small} if no edge stabilizer contains a free subgroup of rank two. We prove:
\begin{theorem}
Let $F$ be a locally compact non-Archimedean field. Let $\rho:\Gamma\to \mathrm{SL}_2(F)$ be a reductive representation.  
If $\rho$ is unbounded, then the induced action of $\Gamma$ on the Bruhat--Tits tree is not small.  
In particular, there exists an edge $e$ such that for every $\gamma\in\Gamma$, the stabilizer of the edge $\gamma. e$ contains a free group of rank two.  
Moreover, if $\rho$ is topologically dense in $\mathrm{SL}_2(F)$, then every edge stabilizer contains a rank-two free subgroup.
\end{theorem}
This follows from Skora’s theorem that a small and minimal action forces the folding map to be an isometry. 

The Hitchin section plays a fundamental role in the study
of character varieties and Teichm\"uller theory \cite{hitchin1987stable,hitchin1992lie,GuichardWienhard12,Labourie06}. In the non-Archimedean setting, it is natural to ask whether an analogue of the Hitchin
section exists, so that we can construct a representation into $\SL_2(F)$ from a Jenkins-Strebel differential.

On the other hand, further characterizations of $\mathrm{im}\Phi$ remains an interesting direction. We therefore propose the following:
\begin{question}
    Given $q\in\mathcal{JS}(\Sigma)$, does there exists $F$ and $\rho:\Gamma\to \mathrm{SL}_2(F)$ such that $\Phi(\rho)=q$? For a fixed $F$, what is $\mathrm{im}\Phi$ as a subset of $\mathcal{JS}(\Sigma)$? 
\end{question}

The recent construction of the real spectrum compactification of character varieties 
\cite{BurgerIozziParreauPozzetti_Currents2021,BurgerIozziParreauPozzettiWeyl2021,BurgerPozzetti2017}
is closely related to these questions, 
suggesting that the non-Archimedean Hitchin fibration may admit a canonical section 
reflecting the geometry of this compactification.

\begin{flushleft}
\textbf{Acknowledgements}: The authors would like to express their gratitude to Hongjie Yu, Bruno Klingler, and Richard Wentworth for many helpful discussions and valuable suggestions during the preparation of this work.
\end{flushleft}

\begin{flushleft}
\textbf{Notations.} 
Throughout this article, unless otherwise specified, $\Sigma$ denotes a compact Riemann surface of genus at least $2$. Its fundamental group is denoted by $\Gamma$, and $p:\tSigma\to\Sigma$ denotes the universal covering.
\end{flushleft}

\section{Preliminaries}
In this section we summarize the preliminary material required for our discussion. We review the notions of quadratic differentials, $\RR$-trees, and non-Archimedean valued fields, and recall the theory of equivariant harmonic maps into $\RR$-trees.

\begin{comment}
In this section, we collect the basic preliminaries for our study.

We begin by reviewing definition of quadratic differentials. Then in Section \ref{Rtree}, we introduce $\RR$-trees and discuss their fundamental properties. 

In Section~\ref{sec:1.3}, we briefly recall the theory of non-Archimedean valued fields and the construction of the associated Bruhat–Tits tree. 

Finally, in Section \ref{sec_harmonicmaps}, we review the theory of equivariant harmonic maps into $\RR$-trees, which is the central technical tool of this article.
\end{comment}

\subsection{Quadratic differentials}

\begin{definition}\label{def_quadraticdiff}
   A holomorphic quadratic differential $q$ on a Riemann surface $\Sigma$ is a global section of $\mathcal{K}_{\Sigma}^{\otimes2}$, where $\mathcal{K}_{\Sigma}$ is the canonical bundle of $\Sigma$. In a complex coordinate $z$ on $\Sigma$, $q$ has an expression $q(z)=\varphi(z)dz^2$, where $\varphi$ is a locally defined holomorphic function. 
\end{definition}

%The relation between quadratic differentials and measured foliations. horizontal foliation. Cite Hubbard-Mazur.

In this paper, whenever we refer to a quadratic differential, we mean a \emph{holomorphic} quadratic differential. A quadratic differential $q$ is called reducible if $q=\omega\otimes\omega$ for some holomorphic $1$-form $\omega$ on $\Sigma$.

A quadratic differential $q$ defines a measured foliation on $\Sigma$. The well-known correspondence between quadratic differentials and measured foliations can be found in \cite{hubbard1979quadratic,WolfRealizing}.

Indeed, consider the $2$-valued closed $1$-form $\Re\sqrt{q}$ on $\Sigma$. This defines a singular foliation $\F_q:=\ker\Re\sqrt{q}$, called the \emph{vertical foliation} of $q$. $\F_q$ is equipped with a \emph{transverse measure} $\mu_q:=|\Re\sqrt{q}|$. For an arc $\gamma$ on $\Sigma$, \[\mu_q(\gamma):=\int_{\gamma}|\langle\Re\sqrt{q},\dot{\gamma}(t)\rangle|dt\] is called the \emph{transverse length} of $\gamma$. Given two points in $\Sigma$ (or two leaves of $\F_q$), their \emph{transverse distance} is defined as the infimum of transverse lengths among arcs connecting them.

\subsection{$\RR$-trees}\label{Rtree}

In this subsection, we introduce basic notions concerning $\RR$-trees. A standard reference is \cite{cullermorgan87groupactions}. 

\begin{definition}
An $\RR$-tree is a metric space $(T,d)$ such that any two points $x,y\in T$ are joined by a unique simple path $[x,y]$, which is isometric to a closed interval in $\RR$. We call $[x,y]$ the geodesic segment between $x$ and $y$.
\end{definition}

\begin{definition}
A point $x\in T$ is an \emph{edge point} if $T\setminus\{x\}$ has exactly two connected components; otherwise $x$ is called a \emph{vertex point}. A geodesic ray is an isometric embedding $\gamma:\RR^+\to T$, and a geodesic line is an isometric embedding $\RR\to T$.

Two geodesic rays in $T$ are said to be equivalent if their intersection is another geodesic ray. An equivalent class of geodesic rays is called an \textbf{end} of $T$. The set of ends is called the boundary of $T$, denoted by $\partial T$.
\end{definition}

\begin{definition}
A $\ZZ$-tree (often called a simplicial tree) is an $\RR$-tree which is also a simplicial $1$-complex. If every vertex has finite valence we say that the tree is of finite valence.
\end{definition}

\begin{definition}\label{def:Gammatree}
     An $\RR$-tree $(T,d)$ with an isometric $\Gamma$-action is called a $\Gamma$-tree. A $\Gamma$-subtree of $T$ is a nonempty, connected $\Gamma$-invariant closed subspace of $T$. Every $\Gamma$-tree contains a unique minimal $\Gamma$-subtree, denoted by $T_{\min}$. And a minimal $\Gamma$-tree is a $\Gamma$-tree containing no proper $\Gamma$-subtrees, i.e., $T=T_{\min}$.
\end{definition}

Given a quadratic differential $q$ on $\Sigma$. Lift the vertical foliation $(\F_q,\mu_q)$ to a $\Gamma$-equivariant measured foliation $(\tilde{\F}_q,\tilde{\mu}_q)$ on the universal cover $\tSigma$. The leaf space $T_q$ of $(\tilde{\F}_q,\tilde{\mu}_q)$ is a quotient space of $\tSigma$ obtained by collapsing leaves to points. It is equipped with a metric $d$ induced by $\tilde{\mu}_q$, as well as a natural isometric $\Gamma$-action induced by the deck transformation. Indeed, $(T_q,d)$ is an $\Gamma$-tree. 

\subsubsection{The length function}\label{section_lengthfunction}

Given a $\Gamma$–tree $(T,d)$,
the associated \emph{length function} is defined by:
\[
\ell_T \colon \Gamma \longrightarrow \mathbb{R}_{\ge 0}, 
\qquad 
\ell_T(\gamma) := \inf_{x\in T} d_T\big(x, \gamma. x\big).
\]

We say that $\ell_T$ is \emph{abelian} if there exists a homomorphism $\mu:\Gamma\to \mathbb{R}$ such that
\[
\ell_T(\gamma) = \big|\mu(\gamma)\big|, \quad \forall \gamma\in\Gamma.
\]
Length functions play a fundamental role in characterizing isometric actions on $\mathbb{R}$–trees, as the following shows.

\begin{lemma}\textup{\cite[1.3]{cullermorgan87groupactions}}\label{lem_ellipticorhyperbolic}
    If $\ell_T(\gamma)=0$, then $\gamma$ has a fixed point in $T$. If $\ell_T(\gamma)\ne 0$, then there exists a geodesic line $\mathrm{Axis}(\gamma)\subset T$, called the axis of $\gamma$, such that \[\mathrm{Axis}(\gamma)=\{x\in T:d(x,\gamma.x)=\ell_T(\gamma)\},\] and $\gamma$ acts on $\mathrm{Axis}(\gamma)$ by translation.
\end{lemma}

Note that if the $\Gamma$-action on $T$ has a global fixed point, then $\ell_T\equiv 0$.

\begin{definition}
    An element $\gamma$ is called elliptic if $\ell_T(\gamma)= 0$ and hyperbolic if $\ell_T(\gamma)\ne 0$.
\end{definition}

\begin{lemma}\textup{\cite[Proposition 3.1]{cullermorgan87groupactions}}\label{lem_minimalsubtree}
    For a $\Gamma$-tree $T$ such that $\ell_T\equiv 0$, there exists a global fixed point of the $\Gamma$-action in $T$, and hence $T_{\min}$ is a single point. If $\ell_T$ is nonzero, then the minimal $\Gamma$-subtree $T_{\min}$ of $T$ is the union of axes of all hyperbolic elements in $\Gamma$, namely, $T_{\min}=\cup_{\ell_T(\gamma)\ne 0}\mathrm{Axis}(\gamma)$.
\end{lemma}

\begin{proposition}\label{prop:length-function}\textup{\cite[Corollary 2.3 and Theorem 3.7]{cullermorgan87groupactions}} 
Let $T$ be a minimal $\Gamma$–tree with nontrivial length function $\ell_T$. The length function $\ell_T$ is nonabelian if and only if $\Gamma$ acts on $T$ without fixed ends. Moreover, if $T'$ is any other minimal $\Gamma$–tree with the same nonabelian length function as $T$, there exists a unique $\Gamma$-equivariant isometry $T \to T'$.
\end{proposition}

\subsection{Bruhat--Tits trees}\label{sec:1.3}

In this subsection, we recall some background on Bruhat-Tits trees for $\SL_2(F)$, where $F$ is a non-Archimedean valued field. Standard references are \cite{serre80trees, bruhattits72}. We begin with a brief review of non-Archimedean valued fields.

\subsubsection{Non-Archimedean valued fields}
A field $F$ is called a \emph{non-Archimedean valued field} if it is equipped with a discrete valuation $v:F^{\times}\to\ZZ$, extended by $v(0)=\infty$, which defines a
non-Archimedean absolute value $|-|$ on $F$. This absolute value takes the form $|-|=c^{v(-)}$ for some $c\in (0,1)$, and satisfies the \emph{ultrametric inequality}:
\[
|x+y|\leq \max\{|x|,|y|\},\qquad \forall x,y\in F.
\]
The associated valuation ring and maximal ideal are
\[
\mathcal O_v=\{x\in F: v(x)\geq 0\}\text{ and  } 
\mathfrak m_v=\{x\in F: v(x)>0\}.
\]
There is an element $\pi\in \mathcal{O}_v$, called a \emph{uniformizer} of $F$, such that $v(\pi)=1$ and $\mathfrak{m}_F=(\pi)$. Denote the \emph{residue field} by $k_v:=\mathcal O_v/\mathfrak m_v$. A non-Archimedean valued field is said to be locally compact if is complete respect to $|-|$ and with residue field $k_v$ finite. If $F$ is locally compact, then it is isomorphic to either a finite extensions of the $p$-adic field $\mathbb Q_p$ (characteristic zero) or Laurent series fields $\mathbb F_{p^r}((t))$ (characteristic $p>0$). 

Throughout this article, unless otherwise specified, $F$ always denotes a non-Archimedean valued field of characteristic zero; it is not assumed to be locally compact or algebraically closed.

\subsubsection{The Bruhat--Tits tree $T_F$ of $\mathrm{SL}_2(F)$}
The Bruhat–Tits tree $T_F$ for $\SL_2(F)$ is a simply connected simplicial complex, serving as the natural non-Archimedean analogue of the symmetric space $\SL_2(\mathbb{R})/\SO(2)$. We briefly recall its construction following \cite{serre80trees}.

A vertex of $T_F$ corresponds to a \emph{homothety} class of $\mathcal O_v$-lattices in $F^2$. 
Two lattices $\Lambda,\Lambda'\subset F^2$ are said to lie in the same homothety class if $\Lambda'=\lambda \Lambda$ for some $\lambda\in F^\times$. Two vertices $[\Lambda]$ and $[\Lambda']$ are joined by an edge if there exist representatives with 
$\pi\Lambda \subsetneq \Lambda' \subsetneq \Lambda$, equivalently if $\Lambda/\Lambda'\cong k_v$. 

For any vertex $[\Lambda]$, the set of vertices adjacent to $[\Lambda]$, i.e., connected by an edge—is in bijection with  
the projective line $\mathbb{P}^1(k_v)$ over the residue field $k_v$.

In particular, when $k_v$ is finite, each vertex has valency $|k_v| + 1$,  
and hence $T_F$ is a $(|k_v| + 1)$-regular tree.

We equip $T_F$ with the path metric in which each edge has length $1$. With this metric, $T_F$ is an example of a $\ZZ$-tree; hence, we will freely use the concepts introduced in Section~\ref{Rtree} for $T_F$ in what follows.

\subsubsection{The $\SL_2(F)$-action on $T_F$}\label{sec_SL2Faction}

We now describe the natural $\SL_2(F)$-action on $T_F$. The group $\SL_2(F)$ acts on $T_F$ by change of basis, sending a homothety of $\mathcal{O}_v$-lattices to another. This action is isometric.

There are exactly two orbits of vertices under this action, and the distance between any two vertices in the same orbit is always even. Indeed, $\SL_2(F)$ acts transitively on the set of edges. See Figure \ref{fig_F2Q3}.

We now describe the stabilizers of various subsets of $T_F$.  
The stabilizer $\mathrm{Stab}([\Lambda])$ of a vertex $[\Lambda]$ is a maximal bounded subgroup of $\SL_2(F)$, conjugate to $\SL_2(\mathcal O_v)$.  
The stabilizer of an edge $e$ is an Iwahori subgroup, conjugate to
\[
\Bigg\{
\begin{pmatrix}a&b\\ c&d\end{pmatrix}\in \SL_2(\mathcal O_v)\;\Big|\; c\in \pi\mathcal O_v
\Bigg\}.
\]
The stabilizer of an end is a Borel subgroup, conjugate to the subgroup of upper triangular matrices in $\SL_2(F)$, while the stabilizer of a geodesic line is a Cartan subgroup, conjugate to the subgroup of diagonal matrices.

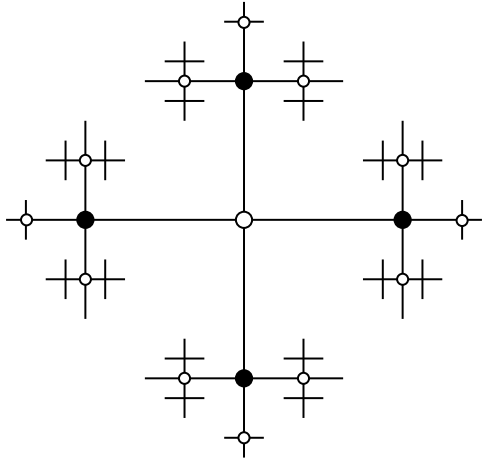
\begin{figure}[!h]
    \centering
\tikzset{every picture/.style={line width=0.75pt}} %set default line width to 0.75pt        

\begin{tikzpicture}[x=0.75pt,y=0.75pt,yscale=-1,xscale=1]

%Straight Lines [id:da8283347620760804] 
\draw    (150,150) -- (390,150) ;
%Straight Lines [id:da033188660497814104] 
\draw    (270,40) -- (270,270) ;
%Straight Lines [id:da9229098395081436] 
\draw    (190,100) -- (190,200) ;
%Straight Lines [id:da8643729925182263] 
\draw    (350,100) -- (350,200) ;
%Straight Lines [id:da8561752389131945] 
\draw    (220,80) -- (320,80) ;
%Straight Lines [id:da6583299391452291] 
\draw    (220,230) -- (320,230) ;
%Straight Lines [id:da7485775685195668] 
\draw    (170,120) -- (210,120) ;
%Straight Lines [id:da2897774147974266] 
\draw    (170,180) -- (210,180) ;
%Straight Lines [id:da04897685059826351] 
\draw    (330,120) -- (370,120) ;
%Straight Lines [id:da856021991012664] 
\draw    (330,180) -- (370,180) ;
%Straight Lines [id:da40540909835023564] 
\draw    (240,60) -- (240,100) ;
%Straight Lines [id:da7097527356730581] 
\draw    (300,60) -- (300,100) ;
%Straight Lines [id:da33585961016737176] 
\draw    (240,210) -- (240,250) ;
%Straight Lines [id:da7805293647782148] 
\draw    (300,210) -- (300,250) ;
%Straight Lines [id:da7134878169619467] 
\draw    (160,140) -- (160,160) ;
%Straight Lines [id:da6825148420978908] 
\draw    (380,140) -- (380,160) ;
%Straight Lines [id:da5171103432028105] 
\draw    (260,50) -- (280,50) ;
%Straight Lines [id:da33636440508711163] 
\draw    (260,260) -- (280,260) ;
%Straight Lines [id:da09315114311307704] 
\draw    (180,110) -- (180,130) ;
%Straight Lines [id:da4729520313349238] 
\draw    (200,110) -- (200,130) ;
%Straight Lines [id:da7912672357230087] 
\draw    (340,110) -- (340,130) ;
%Straight Lines [id:da10118140713750612] 
\draw    (360,110) -- (360,119.83) -- (360,130) ;
%Straight Lines [id:da9974344253246655] 
\draw    (180,170) -- (180,190) ;
%Straight Lines [id:da3227354853808754] 
\draw    (200,170) -- (200,190) ;
%Straight Lines [id:da994935420300323] 
\draw    (340,170) -- (340,190) ;
%Straight Lines [id:da018967509833369367] 
\draw    (360,170) -- (360,190) ;
%Straight Lines [id:da6621491118371335] 
\draw    (230,220) -- (250,220) ;
%Straight Lines [id:da9552234228519072] 
\draw    (290,220) -- (310,220) ;
%Straight Lines [id:da5866656063869911] 
\draw    (290,70) -- (310,70) ;
%Straight Lines [id:da6572513033204296] 
\draw    (290,90) -- (310,90) ;
%Straight Lines [id:da7886568189434958] 
\draw    (230,90) -- (250,90) ;
%Straight Lines [id:da5320812183765852] 
\draw    (230,70) -- (250,70) ;
%Straight Lines [id:da9079768359165318] 
\draw    (230,240) -- (250,240) ;
%Straight Lines [id:da2407184570919726] 
\draw    (290,240) -- (310,240) ;
%Shape: Circle [id:dp8071241694967931] 
\draw  [fill={rgb, 255:red, 255; green, 255; blue, 255 }  ,fill opacity=1 ] (265.92,150) .. controls (265.92,147.74) and (267.74,145.92) .. (270,145.92) .. controls (272.26,145.92) and (274.08,147.74) .. (274.08,150) .. controls (274.08,152.26) and (272.26,154.08) .. (270,154.08) .. controls (267.74,154.08) and (265.92,152.26) .. (265.92,150) -- cycle ;
%Shape: Circle [id:dp6272538186122099] 
\draw  [fill={rgb, 255:red, 0; green, 0; blue, 0 }  ,fill opacity=1 ] (265.92,80) .. controls (265.92,77.74) and (267.74,75.92) .. (270,75.92) .. controls (272.26,75.92) and (274.08,77.74) .. (274.08,80) .. controls (274.08,82.26) and (272.26,84.08) .. (270,84.08) .. controls (267.74,84.08) and (265.92,82.26) .. (265.92,80) -- cycle ;
%Shape: Circle [id:dp015260801276800295] 
\draw  [fill={rgb, 255:red, 0; green, 0; blue, 0 }  ,fill opacity=1 ] (185.92,150) .. controls (185.92,147.74) and (187.74,145.92) .. (190,145.92) .. controls (192.26,145.92) and (194.08,147.74) .. (194.08,150) .. controls (194.08,152.26) and (192.26,154.08) .. (190,154.08) .. controls (187.74,154.08) and (185.92,152.26) .. (185.92,150) -- cycle ;
%Shape: Circle [id:dp2660089694416031] 
\draw  [fill={rgb, 255:red, 0; green, 0; blue, 0 }  ,fill opacity=1 ] (345.92,150) .. controls (345.92,147.74) and (347.74,145.92) .. (350,145.92) .. controls (352.26,145.92) and (354.08,147.74) .. (354.08,150) .. controls (354.08,152.26) and (352.26,154.08) .. (350,154.08) .. controls (347.74,154.08) and (345.92,152.26) .. (345.92,150) -- cycle ;
%Shape: Circle [id:dp3561591790556867] 
\draw  [fill={rgb, 255:red, 0; green, 0; blue, 0 }  ,fill opacity=1 ] (265.92,230) .. controls (265.92,227.74) and (267.74,225.92) .. (270,225.92) .. controls (272.26,225.92) and (274.08,227.74) .. (274.08,230) .. controls (274.08,232.26) and (272.26,234.08) .. (270,234.08) .. controls (267.74,234.08) and (265.92,232.26) .. (265.92,230) -- cycle ;
%Shape: Circle [id:dp9647216985466843] 
\draw  [fill={rgb, 255:red, 255; green, 255; blue, 255 }  ,fill opacity=1 ] (347.21,120) .. controls (347.21,118.46) and (348.46,117.21) .. (350,117.21) .. controls (351.54,117.21) and (352.79,118.46) .. (352.79,120) .. controls (352.79,121.54) and (351.54,122.79) .. (350,122.79) .. controls (348.46,122.79) and (347.21,121.54) .. (347.21,120) -- cycle ;
%Shape: Circle [id:dp0579333183998646] 
\draw  [fill={rgb, 255:red, 255; green, 255; blue, 255 }  ,fill opacity=1 ] (347.21,180) .. controls (347.21,178.46) and (348.46,177.21) .. (350,177.21) .. controls (351.54,177.21) and (352.79,178.46) .. (352.79,180) .. controls (352.79,181.54) and (351.54,182.79) .. (350,182.79) .. controls (348.46,182.79) and (347.21,181.54) .. (347.21,180) -- cycle ;
%Shape: Circle [id:dp43845315514936534] 
\draw  [fill={rgb, 255:red, 255; green, 255; blue, 255 }  ,fill opacity=1 ] (297.21,230) .. controls (297.21,228.46) and (298.46,227.21) .. (300,227.21) .. controls (301.54,227.21) and (302.79,228.46) .. (302.79,230) .. controls (302.79,231.54) and (301.54,232.79) .. (300,232.79) .. controls (298.46,232.79) and (297.21,231.54) .. (297.21,230) -- cycle ;
%Shape: Circle [id:dp4739410415206303] 
\draw  [fill={rgb, 255:red, 255; green, 255; blue, 255 }  ,fill opacity=1 ] (237.21,230) .. controls (237.21,228.46) and (238.46,227.21) .. (240,227.21) .. controls (241.54,227.21) and (242.79,228.46) .. (242.79,230) .. controls (242.79,231.54) and (241.54,232.79) .. (240,232.79) .. controls (238.46,232.79) and (237.21,231.54) .. (237.21,230) -- cycle ;
%Shape: Circle [id:dp07017425407988132] 
\draw  [fill={rgb, 255:red, 255; green, 255; blue, 255 }  ,fill opacity=1 ] (187.21,120) .. controls (187.21,118.46) and (188.46,117.21) .. (190,117.21) .. controls (191.54,117.21) and (192.79,118.46) .. (192.79,120) .. controls (192.79,121.54) and (191.54,122.79) .. (190,122.79) .. controls (188.46,122.79) and (187.21,121.54) .. (187.21,120) -- cycle ;
%Shape: Circle [id:dp9624799338347448] 
\draw  [fill={rgb, 255:red, 255; green, 255; blue, 255 }  ,fill opacity=1 ] (187.21,180) .. controls (187.21,178.46) and (188.46,177.21) .. (190,177.21) .. controls (191.54,177.21) and (192.79,178.46) .. (192.79,180) .. controls (192.79,181.54) and (191.54,182.79) .. (190,182.79) .. controls (188.46,182.79) and (187.21,181.54) .. (187.21,180) -- cycle ;
%Shape: Circle [id:dp3423676260053131] 
\draw  [fill={rgb, 255:red, 255; green, 255; blue, 255 }  ,fill opacity=1 ] (297.21,80) .. controls (297.21,78.46) and (298.46,77.21) .. (300,77.21) .. controls (301.54,77.21) and (302.79,78.46) .. (302.79,80) .. controls (302.79,81.54) and (301.54,82.79) .. (300,82.79) .. controls (298.46,82.79) and (297.21,81.54) .. (297.21,80) -- cycle ;
%Shape: Circle [id:dp11846165174984835] 
\draw  [fill={rgb, 255:red, 255; green, 255; blue, 255 }  ,fill opacity=1 ] (237.21,80) .. controls (237.21,78.46) and (238.46,77.21) .. (240,77.21) .. controls (241.54,77.21) and (242.79,78.46) .. (242.79,80) .. controls (242.79,81.54) and (241.54,82.79) .. (240,82.79) .. controls (238.46,82.79) and (237.21,81.54) .. (237.21,80) -- cycle ;
%Shape: Circle [id:dp9270364803559723] 
\draw  [fill={rgb, 255:red, 255; green, 255; blue, 255 }  ,fill opacity=1 ] (267.21,50.33) .. controls (267.21,48.79) and (268.46,47.54) .. (270,47.54) .. controls (271.54,47.54) and (272.79,48.79) .. (272.79,50.33) .. controls (272.79,51.88) and (271.54,53.13) .. (270,53.13) .. controls (268.46,53.13) and (267.21,51.88) .. (267.21,50.33) -- cycle ;
%Shape: Circle [id:dp41825593635093916] 
\draw  [fill={rgb, 255:red, 255; green, 255; blue, 255 }  ,fill opacity=1 ] (157.54,150) .. controls (157.54,148.46) and (158.79,147.21) .. (160.33,147.21) .. controls (161.88,147.21) and (163.13,148.46) .. (163.13,150) .. controls (163.13,151.54) and (161.88,152.79) .. (160.33,152.79) .. controls (158.79,152.79) and (157.54,151.54) .. (157.54,150) -- cycle ;
%Shape: Circle [id:dp2851394567556317] 
\draw  [fill={rgb, 255:red, 255; green, 255; blue, 255 }  ,fill opacity=1 ] (267.21,260) .. controls (267.21,258.46) and (268.46,257.21) .. (270,257.21) .. controls (271.54,257.21) and (272.79,258.46) .. (272.79,260) .. controls (272.79,261.54) and (271.54,262.79) .. (270,262.79) .. controls (268.46,262.79) and (267.21,261.54) .. (267.21,260) -- cycle ;
%Shape: Circle [id:dp9068105940513242] 
\draw  [fill={rgb, 255:red, 255; green, 255; blue, 255 }  ,fill opacity=1 ] (377.21,150.33) .. controls (377.21,148.79) and (378.46,147.54) .. (380,147.54) .. controls (381.54,147.54) and (382.79,148.79) .. (382.79,150.33) .. controls (382.79,151.88) and (381.54,153.13) .. (380,153.13) .. controls (378.46,153.13) and (377.21,151.88) .. (377.21,150.33) -- cycle ;

\end{tikzpicture}

    \caption{Two-level Bruhat--Tits tree for $\mathrm{SL}_2(\mathbb{Q}_3)$ (each vertex has $4$ neighbors).}
    \label{fig_F2Q3}
\end{figure}

The boundary $\partial T_F$ (the set of ends) of $T_F$ can be naturally identified with the projective line $\mathbb{P}^1(F)$ over the valued field $F$. The $\SL_2(F)$-action on $T_F$ induces an action on the boundary $\partial T_F$, which coincides with the standard $\SL_2(F)$-action on $\mathbb{P}^1(F)$. For details, see \cite[Chapter II.1.1]{serre80trees} or \cite[Appendix B.4.2]{Kahlergroups_Py}.
\begin{lemma}\label{lem_eigen}
For any $A\in\SL_2(F)$, the following are equivalent:
\begin{itemize}
    \item [\emph{(i)}] $A$ fixes an end of $T_F$;
    \item [\emph{(ii)}] $A$ has a fixed point in $\mathbb{P}^1(F)$;
    \item [\emph{(iii)}] $A$ admits an eigenvector in $F^2$.
\end{itemize}
\end{lemma}

\begin{lemma}\textup{\cite[Proposition II.3.15]{morganshalen1984valuationstrees}}\label{lem_length=logtrace}
    For any $A\in\SL_2(F)$, we have 
    \begin{align}\label{formula_length=logtrace}
        \ell_{T_F}(A)=-2\min\{0,v(\trace A)\}.
    \end{align}
\end{lemma}

By Lemma \ref{lem_ellipticorhyperbolic}, if $A$ is elliptic, then $A$ fixes some vertex $[\Lambda]\in T_F$, and hence conjugate to a matrix in $\SL_2(\mathcal{O}_v)$. If $A\in\SL_2(F)$ is hyperbolic, then $A$ preserves a geodesic line in $T_F$. Therefore, by Lemma \ref{lem_eigen}, $A$ is diagonalizable and conjugate to $\begin{pmatrix}a&0\\ 0&a^{-1}\end{pmatrix}$ for some $a\in F^{\times}$. The length of $A$ is $-2v(a+a^{-1})=-2v(\trace A)$.

\begin{comment}
	The Bruhat--Tits tree can be viewed as a skeleton of the Berkovich projective line, and provides a geometric framework for studying $p$-adic surface group representations, harmonic maps, and compactifications of character varieties in the style of Culler--Shalen theory.
\end{comment}

\subsection{Harmonic maps into $\RR$-trees}\label{sec_harmonicmaps}

In this subsection, we define equivariant harmonic maps from the universal cover $\tSigma$ of a compact Riemann surface $\Sigma$ into $\RR$-trees, and present several relations between the harmonic map and the length function associated to the $\RR$-tree.

\subsubsection{Equivariant harmonic maps into $\RR$-trees}

 Let $T$ be an $\RR$-tree with an isometric action of $\Gamma$, i.e. a $\Gamma$-tree. A map $u:\tSigma\to T$ is called \emph{$\Gamma$-equivariant} if
\[
u(\gamma. x) = \gamma. u(x), \qquad \forall \gamma\in\Gamma, \; x\in\tSigma.
\]

Following \cite{korevaar1993sobolev}, the energy density of $u\in W^{1,2}_{\mathrm{loc}}(\tSigma,T)$ is defined by
\[
e(u)(x) = \limsup_{r\to 0} \frac{1}{r^2} \fint_{B_r(x)} d\!\big(u(y),u(x)\big)^2\,dy,
\]
and the total energy is $E(u)=\int_\Sigma e(u)\,d\mathrm{vol}_g$. A $\Gamma$-equivariant map is called \emph{harmonic} if it minimizes energy under compactly supported $\Gamma$-equivariant variations.

A point $x\in\tSigma$ is called \emph{regular} if $u$ maps a neighborhood of $x$ into a geodesic segment of $T$; otherwise $x$ is called \emph{singular}. By \cite{GromovSchoen1992,sun2003regularity}, the singular set has Hausdorff dimension zero.

Associated with the harmonic map $u$ is a $\Gamma$-equivariant quadratic differential $4\partial u\otimes\partial u$ on $\tSigma$. In local coordinates $z$, it can be written as $4(u_z)^2\,dz^2$. This descends to a quadratic differential $q_u$ on $\Sigma$, known as \emph{the Hopf differential} of $u$. 

\begin{theorem}\textup{\cite[Theorem 4.4]{daskalopoulos1998character}}\label{thm_minimal}
	Let $\Sigma$ be a compact Riemann surface, $u:\tSigma\to T$ be a $\Gamma$-equivariant harmonic map. The image of $u$ is the minimal $\Gamma$-subtree $T_{\min}$ of $T$.
\end{theorem}

\begin{theorem}[{\cite{GromovSchoen1992,korevaari1997global}}]
\label{theorem_existence_uniqueness}
Let $T$ be a $\Gamma$-tree. If $u_0,u_1:\tSigma\to T$ are two distinct $\Gamma$-equivariant harmonic maps, then their images are contained in a common geodesic of $T$, and their Hopf differentials coincide. If $\ell_{T}\equiv 0$, then the harmonic map must be constant. If $\ell_{T}$ is nontrivial and nonabelian, then the $\Gamma$-equivariant harmonic map is unique. 
\end{theorem}
\begin{remark}
    If $\ell_{T}$ is abelian, a $\Gamma$-equivariant harmonic map from $\tSigma$ to $T$ may not exist. However, if such an equivariant harmonic map $u:\tSigma\to T$ does exist, then the image of $u$, which is the minimal $\Gamma$-subtree of $T$, must be a geodesic line. Consequently, in this case, the existence of harmonic maps implies that the $\Gamma$-action on $T$ fixes two ends. See \cite[Lemma 6.1]{hewentworthzhang24}.
\end{remark}

\subsection{Hopf differential and folding map}
\noindent
We now discuss the relationship between the Hopf differential and the image tree of a harmonic map.  

Let $T$ be a $\Gamma$-tree and let $u:\widetilde{\Sigma}\to T$ be a $\Gamma$-equivariant harmonic map, with $q_u$ denoting its Hopf differential. 
Recall that the leaf space of $q_u$ defines a $\Gamma$-tree $T_{q_u}$, as introduced in Section~\ref{Rtree}. Consider the projection map $\pi:\widetilde{\Sigma}\to T_{q_u}$, which collapses each leaf to a point. 
It was proved in \cite{wolf1995harmonic} and \cite[Proposition~2.2]{daskalopoulos2000morganshalen} that $\pi$ is itself a harmonic map. 

To proceed, we introduce morphisms between $\Gamma$-trees.

\begin{definition}
Let $T$ and $T'$ be $\Gamma$-trees. 
A \emph{morphism} $f : T \to T'$ is a $\Gamma$–equivariant, continuous, 
piecewise isometric map that sends each edge of $T$ to an edge path in $T'$. 
We say that $f$ \emph{folds at a point} $x \in T$ if there exist segments 
$[x,y_1]$ and $[x,y_2]$ with $[x,y_1] \cap [x,y_2] = \{x\}$ 
such that $f$ maps both $[x,y_1]$ and $[x,y_2]$ isometrically onto a common segment in $T'$. 
A morphism that folds at some point is called a \emph{folding map}.
\end{definition}

\begin{remark}
By \cite[Lemma~I.1.1]{morganotal93}, a morphism is an embedding (hence an isomorphism onto its image) unless it folds at some point. 
Moreover, folding maps can be quite complicated: they may map vertex points to interior points of edges and conversely. 
We refer to \cite[Section~3.1.2]{Daskalopoulos-Wentworth05} for a more detailed discussion.
\end{remark}

The following factorization result holds.

\begin{proposition}[{\cite[Proposition 2.4]{daskalopoulos2000morganshalen}}]
\label{prop_harmoncmap_factortree}
The harmonic map $u:\widetilde{\Sigma}\to T$ factors as $u=f\circ \pi$, 
where $\pi$ is the projection to the Hopf differential tree $T_{q_u}$ of $u$ and 
$f:T_{q_u}\to T$ is a folding map.
\end{proposition}

\section{The non-Archimedean Character variety and the Hitchin map}
This section is devoted to the construction of the non-Archimedean character variety and the Hitchin map. After introducing the representation variety and its non-Archimedean topology, we use the theory of equivariant harmonic maps into $\RR$-trees to show that the induced Hitchin map is continuous.

\subsection{Types of representations, the length function, and the trace}\label{sec_3.1}
In this subsection, we define several properties of representations and their relationships to the length function and the trace function. Throughout this subsection, $\rho$ denotes a representation $\rho:\Gamma\to\SL_2(F)$. The results presented in the following two subsections hold for any finitely generated group, not just surface groups.

\begin{definition}
A representation $\rho:\Gamma\to \SL_2(F)$ is called \textbf{irreducible} if it has no invariant lines (one-dimensional subspaces) in $F^2$; otherwise it is \textbf{reducible}. A reducible representation is \textbf{completely reducible} if it has two distinct invariant lines in $F^2$.
\end{definition}

\begin{definition}
    A representation $\rho:\Gamma\to\SL_2(F)$ define a $\Gamma$-action on the Bruhat-Tits tree $T_F$. We define the length function $\ell_{\rho}$ of the $\rho$ by \[\ell_{\rho}:\Gamma\to\ZZ_{\ge 0}:\gamma\mapsto \ell_{T_F}(\rho(\gamma)).\] 
\end{definition}

\begin{definition}
    $\rho$ is called bounded if $\rho(\Gamma)$ is a bounded subgroup of $\SL_2(F)$ with respect to the non-Archimedean norm. It is called unbounded if it is not bounded.
\end{definition}

\begin{proposition}\label{prop:bounded-fixed-vertex}
The following are equivalent:
\begin{itemize}
    \item[\emph{(i)}]  $\rho$ is bounded;
    \item[\emph{(ii)}] $\rho$ is conjugate to a representation $\rho':\Gamma\to\SL_2(\mathcal{O}_v)$;
    \item[\emph{(iii)}] $\ell_\rho\equiv 0$;
    \item[\emph{(iv)}] The corresponding $\Gamma$-action fixed a vertex $[\Lambda]\in T_F$.
\end{itemize}
\end{proposition}
\begin{proof}
The equivalence of (ii), (iii) and (iv) follows from Lemma \ref{lem_ellipticorhyperbolic}. The equivalence between (i) and (ii) follows from the fact that $\SL_2(\mathcal{O}_v)$ is a maximal bounded subgroup of $\SL_2(F)$.
\end{proof}

\begin{lemma}
    Suppose $\rho$ is unbounded. Then it is reducible if and only if $\ell_{\rho}$ is abelian.
\end{lemma}
\begin{proof}
    If $\rho$ is unbounded, then by Proposition \ref{prop:bounded-fixed-vertex}, $\ell_{\rho}\ne 0$. By Lemma \ref{lem_eigen} and Proposition \ref{prop:length-function}, $\rho$ is reducible if and only if it fixes an end of $T_F$, which holds if and only if $\ell_{\rho}$ is abelian.
\end{proof}

As a result, for an unbounded representation $\rho$, it is irreducible if and only if $\ell_{\rho}$ is nonabelian.

We have the following characterization of unbounded irreducible representations:
\begin{lemma}\textup{\cite[Theorem 2.7]{cullermorgan87groupactions}}\label{lem_freegrpinnonabelian}
    An unbounded representation $\rho$ is irreducible if and only if there exists two elements $\gamma_1,\gamma_2\in \Gamma$ that generate a free subgroup of rank $2$ in $\Gamma$, and that $\langle\gamma_1,\gamma_2\rangle$ acts freely and properly discontinuously on $T_F$.
\end{lemma}

\subsubsection{The trace function}
The following lemma is essentially due to \cite[Proposition 1.4.1]{CullerShalen83} and \cite[Lemma 4]{invitationCullerShalen}. Note that these references work over algebraically closed fields; however, their proofs rely only on the Cayley–Hamilton theorem, which holds over any commutative ring with identity.
\begin{lemma}\label{lem_tracepolynomial}
The trace of any word in the matrices $A_1,\dots,A_n \in \SL_2(F)$ is a polynomial with integral coefficients in the $2^n-1$ traces $\trace(A_{j_1}\cdots A_{j_m})$, with $1\le j_1<\cdots<j_m\le n$ and $m\le n$.
\end{lemma}

\begin{definition}
    For a representation $\rho$, the trace function is defined as \[\trace_\rho:\Gamma\to F:\gamma\mapsto \trace(\rho(\gamma)).\]
\end{definition}

We deduce from Proposition \ref{prop:bounded-fixed-vertex} that:
\begin{proposition}
    A representation $\rho$ is bounded if and only if $\trace_{\rho}$ takes values in $\mathcal{O}_v$.
\end{proposition}

\begin{definition}
    A representation $\rho:\Gamma\to\SL_2(F)$ is called absolutely irreducible, if it is irreducible and the induced representation $\bar{\rho}:\Gamma\to\SL_2(\overline{F}):\gamma\mapsto \rho(\gamma)$ remains irreducible over the algebraic closure $\overline{F}$.
\end{definition}

\begin{lemma}\textup{\cite[Lemma 1.2.1]{CullerShalen83}}\label{lem_reduciblechar}
    Let $\rho:\Gamma\to\SL_2(F)$ be a representation with nonabelian image. Then the following are equivalent:
    \begin{itemize}
        \item[\emph{(i)}] $\rho$ is reducible;
        \item[\emph{(ii)}] $\bar{\rho}$ is reducible over $\overline{F}$;
        \item[\emph{(iii)}] $\trace_{\rho}(\gamma)=2$ for each element $\gamma$ of the commutator subgroup $[\Gamma,\Gamma]$.
    \end{itemize}
\end{lemma}

\begin{proposition}\label{prop_absirr}
    An unbounded irreducible representation $\rho:\Gamma\to\SL_2(F)$ is absolutely irreducible. 
\end{proposition}
\begin{proof}
    By Lemma \ref{lem_freegrpinnonabelian}, $\rho(\Gamma)$ contains a nonabelian free subgroup and hence Lemma \ref{lem_reduciblechar} applies to $\rho$. Thus $\bar{\rho}$ is irreducible over $\overline{F}$.
\end{proof}

\begin{lemma}\textup{\cite[Theorem 6.12]{Nakamoto00}}\label{lem_trace_determines_absirr}
    Suppose $\rho:\Gamma\to\SL_2(F)$ is absolutely irreducible, and let $\rho'$ be another representation. Then $\rho$ is conjugate to $\rho'$ over $F$ if and only if $\trace_{\rho}\equiv\trace_{\rho'}$.
\end{lemma}

The following is the direct consequence of Proposition \ref{prop_absirr} and Lemma \ref{lem_trace_determines_absirr}.
\begin{proposition}\label{prop_trace_determines_ubirr}
    Let $\rho$ be an unbounded irreducible representation, and let $\rho'$ be another representation. Then $\rho$ is conjugate to $\rho'$ over $F$ if and only if $\trace_{\rho}\equiv \trace_{\rho'}$.
\end{proposition}

\subsection{The Representation variety}\label{sec:3.2}
In this subsection, we define the non-archimedean Representation variety and two different topologies on it.

Let $\gamma_1,\dots,\gamma_n\in\Gamma$ be a set of generators. A representation $\rho:\Gamma\to \SL_2(\Gamma)$ is uniquely determined by the images $\rho(\gamma_1),\dots,\rho(\gamma_n)$. Therefore, there is an injection \[\phi_{\rho,\gamma}:\mathrm{Hom}(\Gamma,\SL_2(F))\to F^{4n}:\rho\mapsto (\rho(\gamma_1),\dots,\rho(\gamma_n)).\] It turns out that the image $\mR_F(\Gamma):=\phi_{\rho,\gamma}(\mathrm{Hom}(\Gamma,\SL_2(F)))$ is an algebraic set in $F^{4n}$, equipped with the Zariski topology. 

If we choose another set of generators $\delta_1,\dots,\delta_m$, then there is another injection $\phi_{\rho,\delta}$ mapping $\mathrm{Hom}(\Gamma,\SL_2(F))$ to another algebraic set $\mR'_F(\Gamma)$ in $F^{4m}$. We then obtain a bijection \[\phi_{\rho,\delta}\circ\phi_{\rho,\gamma}^{-1}:\mR_F(\Gamma)\to\mR'_F(\Gamma).\]
Since each $\delta_{j}$ can be expressed as words in $\gamma_i$'s,  $\phi_{\rho,\delta}\circ\phi_{\rho,\gamma}^{-1}$ is a restriction of a polynomial map $\phi=(\phi_1,\dots,\phi_m):F^n\to F^m$, where each $\phi_k$ is a polynomial in $4n$ variables over $F$. The same holds for the inverse $\phi_{\rho,\gamma}\circ\phi_{\rho,\delta}^{-1}$. Thus $\mR_F(\Gamma)$ and $\mR'_F(\Gamma)$ are isomorphic as algebraic sets over $F$.

On the other hand, recall that there is a non-Archimedean norm $|-|$ on $F$. This induces a non-Archimedean topology on $F^{4n}$, and since polynomials are continuous in this topology, $\mR_F(\Gamma)$ is closed in $F^{4n}$. The induced subspace topology on $\mR_F(\Gamma)$ is also called the \emph{non-Archimedean topology}.

Moreover, by continuity of polynomial maps again, for any other generating set $\delta_1,\dots,\delta_m$, the bijection $\phi_{\rho,\delta}\circ\phi_{\rho,\gamma}^{-1}$ is a homeomorphism with respect to the non-Archimedean topology. Therefore, the non-Archimedean topology on the representation variety is independent of the choice of generators.

Therefore, from now on, we fix a generating set $\gamma_1,\dots,\gamma_n$ and define \emph{the representation variety of $\Gamma$} to be the set $\mR_F(\Gamma)$. A point in $\mR_F(\Gamma)$ is viewed both as a representation and as a vector in $F^{4n}$. 

In summary, there are two topologies on $\mR_F(\Gamma)$: the Zariski topology and the non-Archimedean topology. Both of them are independent of the choice of generators, and the former is coarser than the latter.

\subsubsection{Two components of $\mR_F(\Gamma)$}

Let $\mR^b_F(\Gamma)\subset\mR_F(\Gamma)$ be the subset of bounded representations, and let $\mR^{ub}_F(\Gamma)$ be the subset of unbounded representations. They are disjoint by definition. Moreover, we have:
\begin{proposition}\label{prop_openness}
    Both $\mR^b_F(\Gamma)$ and $\mR^{ub}_F(\Gamma)$ are open subsets of $\mR_F(\Gamma)$ under the non-Archimedean topology.
\end{proposition}
 To prove this proposition, we need the following standard result for non-Archimedean field:

\begin{lemma}\label{lem_open_and_closed}
    The valuation ring $\mathcal{O}_{v}$ is both open and closed in $F$ with respect to the non-Archimedean topology.
\end{lemma}
\begin{proof}
    Indeed, $\mathcal{O}_v=\{x\in F:v(x)\ge 0\}$ is the unit ball $\{x\in F: |x|\le 1\}$, hence it is closed. To see that it is also open, recall that the norm $|-|$ satisfying $|x+y|\le \max\{|x|,|y|\}$ for all $x,y\in F$. If $x\in \mathcal{O}_v$ and $|x|=1$, then we can choose an open ball $B_{1/2}(x)=\{x+y\in F:|y|<1/2\}$ centered at $y$. Then $\forall x+y\in B_{1/2}(x)$, we have $|x+y|\le 1$, hence $B_{1/2}(x)\subset \mathcal{O}_v$. Therefore, $\mathcal{O}_v$ is open.
\end{proof}

We also need the following lemma:
\begin{lemma}\label{lem_elementtraceiscontinuous}
    Fix an element $\gamma\in\Gamma$. The function \[\trace_{-}(\gamma):\mR_F(\Gamma)\to F:\rho\mapsto \trace_{\rho}(\gamma)\] is continuous with respect to both the Zariski topology and the non-Archimedean topology.
\end{lemma}
\begin{proof}
    Since $\gamma$ can be expressed as a word in the generators $\gamma_1,\dots,\gamma_n$, it follows from Lemma \ref{lem_tracepolynomial}, $\trace_{\rho}(\gamma)$ is a polynomial with integral coefficients in the $2^n-1$ traces \[\trace(\rho(\gamma_{j_1})\cdots \rho(\gamma_{j_m})),\,\, 1\le j_1<\cdots<j_m\le n,\,\, m\le n.\]

    And each function $\trace(\rho(\gamma_{j_1})\cdots \rho(\gamma_{j_m}))$ is a polynomial in the matrix entries of $\rho(\gamma_i)~(1\le i\le n)$, hence is continuous in both the Zariski topology and the non-Archimedean topology. This finishes the proof.
\end{proof}

\begin{proof}[Proof of Proposition \ref{prop_openness}]
    By Proposition \ref{prop:bounded-fixed-vertex}, if $\rho\in\mR^b_F(\Gamma)$, then there exists $A\in \SL_2(F)$, such that $\rho'=A\rho A^{-1}$ takes values in $\SL_2(\mathcal{O}_v)$. Thus $(\rho'(\gamma_i))_{1\le i\le n}\in\mathcal{O}_v^{4n}\subset F^{4n}$. By Lemma \ref{lem_open_and_closed}, $\mathcal{O}_v^{4n}$ is open in $F^{4n}$, so the intersection $\mathcal{O}_v^{4n}\cap \mR_F(\Gamma)=:\mathcal{U}$ is an open neighborhood of $\rho'$. 
    
    Now, every representation $\rho''\in\mathcal{U}$ satisfies that $\rho''(\gamma_i)\in\SL_2(\mathcal{O}_v)$ for each generator $\gamma_i$, thus fixes the standard $\mathcal{O}_v$-lattice $[\Lambda_0]$; hence $\rho''$ must be bounded. Thus $\mathcal{U}\subset \mR^b_{F}(\Gamma)$, and $A^{-1}\mathcal{U}A$ is a neighborhood of $\rho$ in $\mR^b_F(\Gamma)$. This shows that $\mR^b_F(\Gamma)$ is open in the non-Archimedean topology.

    Suppose $\rho$ is unbounded, then there exists an element $\gamma\in\Gamma$, such that $\trace_{\rho}(\gamma)\notin\mathcal{O}_v$. By Lemma \ref{lem_elementtraceiscontinuous} and the fact that $\mathcal{O}_v$ is closed, there exits an open neighborhood $\mathcal{V}$ of $\rho$ in $\mR_F(\Gamma)$ such that $\trace_{-}(\gamma)(\mathcal{U})$ is contained in $F\setminus\mathcal{O}_v$. 

    If $\rho'\in\mathcal{V}$, then $v(\trace(\rho'(\gamma)))<0$, so by \eqref{formula_length=logtrace}, $\ell_{\rho'}(\gamma)>0$. Thus $\rho'$ is unbounded and hence $\mathcal{V}\subset \mR_F^{ub}(\Gamma)$. This completes the proof.
\end{proof}

\subsection{The non-Archimedean Hitchin map}\label{sec:3.3}
Consider a compact Riemann surface $\Sigma$ and its fundamental group $\Gamma$. In this section we will define a map \[\overline{\Phi}:\mR_F(\Gamma)\to H^0(\Sigma,\mathcal{K}_{\Sigma}^{\otimes 2}).\]

\begin{definition}
    Consider the space $\RR_{\ge 0}^{\Gamma}$ of nonnegative real-valued functions on $\Gamma$. Let $\mathcal{L}(\Gamma)$ denote the subspace consisting of those $f\in\RR_{\ge 0}^{\Gamma}$ for which exists a $\Gamma$-tree $T$ such that $\ell_T\equiv f$. The space $\RR_{\ge 0}^{\Gamma}$ is equipped with the weak topology and $\mathcal{L}(\Gamma)$ is equipped with the subspace topology.
\end{definition}

\begin{proposition}\textup{\cite[Proposition 3.6]{daskalopoulos2000morganshalen}}\label{prop:abelianHopfdiff}
    Let $\ell\in\mathcal{L}(\Gamma)$ be a nontrivial abelian length function. Then there exists a $\Gamma$-action on $\RR$ by translations, such that the associated length function satisfying $\ell_{\RR}\equiv\ell$. Moreover, there exists a harmonic map $u:\tSigma\to \RR$, with reducible Hopf differential $q_u=\omega\otimes\omega$, where $\omega$ is a holomorphic $1$-form on $\Sigma$. Furthermore, we have $\ell(\gamma)=|\int_{\gamma}\Re\omega|$.
\end{proposition}

We define a map $H:\mathcal{L}(\Gamma)\setminus\{0\}\to H^0(\Sigma,\mathcal{K}_{\Sigma}^{\otimes 2})$ as follows.

If $\ell$ is nonabelian, then up to $\Gamma$-equivariant isometry there exists a unique $\Gamma$-tree $T$ with length function $\ell$, and $H(\ell)$ is the Hopf differential of the unique $\Gamma$-equivariant harmonic map $u:\tSigma\to T$. 
If $\ell$ is abelian, we set $H(\ell)=\omega\otimes\omega$, where $\omega$ is the holomorphic $1$-form provided by Proposition~\ref{prop:abelianHopfdiff}.

\begin{lemma}\textup{\cite[Theorem 3.9]{daskalopoulos2000morganshalen}}\label{lem:Hiscontinuous}
    The map $H:\mathcal{L}(\Gamma)\setminus\{0\}\to H^0(\Sigma,\mathcal{K}_{\Sigma}^{\otimes 2})$ is continuous.
\end{lemma}

Next, we define: \begin{align}\label{formula:Psi}\Psi:\mR_F(\Gamma)\to \mathcal{L}(\Gamma):\rho\mapsto\ell_{\rho}.\end{align}

\begin{definition}\label{def_repHitchin}
    The map $\overline{\Phi}:\mR_F(\Gamma)\to H^0(\Sigma,\mathcal{K}_{\Sigma}^{\otimes 2})$ is defined as $H\circ\Psi$ on the space of unbounded representations $\mR^{ub}_F(\Gamma)$. On the space of bounded representations $\mR^{b}_F(\Gamma)$, $\overline{\Phi}$ is defined to be the zero map, sending every bounded representation to the zero quadratic differential on $\Sigma$.
\end{definition}

\begin{theorem}\label{thm:preconti}
    Let $\mR_F(\Gamma)$ be equipped with the non-Archimedean topology and $H^0(\Sigma,\mathcal{K}_{\Sigma}^{\otimes 2})$ bequipped with the Euclidean topology. Then $\overline{\Phi}$ is continuous. 
\end{theorem}
\begin{proof}
    Since a constant map is always continuous, by Proposition \ref{prop_openness}, it suffices to show that $\overline{\Phi}|_{\mR_F^{ub}(\Gamma)}$ is continuous.

    By Lemma \ref{lem:Hiscontinuous}, we only need to show that $\Psi:\mR^{ub}_F(\Gamma)\to \mathcal{L}(\Gamma)\setminus\{0\}$ is continuous. Since $\RR_{\ge 0}^{\Gamma}$ is equipped with the weak topology, this reduces to show that for each $\gamma\in\Gamma$, the map
    \[\rho=(\rho(\gamma_1),\dots,\rho(\gamma_n))\mapsto \ell_{\rho}(\gamma)\] is continuous in the matrix entries of $\rho(\gamma_i)$'s.

    Recall from \eqref{formula_length=logtrace} that $\ell_{\rho}(\gamma)=-2\min\{0,v(\trace(\rho(\gamma)))\}$. By Lemma \ref{lem_elementtraceiscontinuous}, the map \[(\rho(\gamma_1),\dots,\rho(\gamma_n))\mapsto\trace(\rho(\gamma))\] is continuous in the matrix entries. It remains to show that the map\[\mathfrak{L}:F\to \ZZ:x\mapsto \min\{0,v(x)\}\] is continuous with respect to the non-Archimedean norm. 

    On $\mathcal{O}_v$, we have $\mathfrak{L}|_{\mathcal{O}_v}\equiv 0$, hence $\mathfrak{L}$ is continuous on ${\mathcal{O}_v}$. Since $\mathcal{O}_v$ is both open and closed, it suffices to verify continuity on $F\setminus\mathcal{O}_v$. But $\mathfrak{L}=v$ on $F\setminus\mathcal{O}_v$, hence must be continuous. 
\end{proof}

\subsection{The character variety}\label{sec:3.4}
In this subsection, we define the character variety $\mathcal{X}_F(\Gamma)$, and show that the Hitchin map $\overline{\Phi}$ descends to $\mathcal{X}_{F}(\Gamma)$. Since the valued field $F$ is not assumed to be algebraically closed, $\mathcal{X}_F(\Gamma)$ does not naturally carry the structure of an algebraic variety. 

The algebraic group $\SL_2(F)$ acts on $\mR_F(\Gamma)$ by conjugation. One may consider the \emph{closure equivalence relation} on $\mR_F(\Gamma)$ as follows: two representations $\rho_1,\rho_2\in \mR_F(\Gamma)$ are said to be \emph{closure equivalent}, denoted $\rho_1\sim_c\rho_2$, if and only if the \emph{Zariski closures} of their orbits $[\rho_1]$ and $[\rho_2]$ intersect.

However, it does not immediately follow from this definition that $\sim_c$ is an equivalence relation. While reflexivity and symmetry are clear, \emph{transitivity} is not guaranteed. 

Instead, we use the following abstract construction:
\begin{definition}\cite[Definition 4.2.1]{noteonchar}
   Let $X$ be a general topological space. Define an equivalence relation $x\sim y$ on $X$ by setting $x\sim y$ if and
only if $x \approx y$ for all equivalence relations $\approx$ on $X$ such that $X/\approx$ is Hausdorff. The quotient $\mathrm{Hau}(X):=X/\sim$ is Hausdorff and called the \textbf{Hausdorffization} of X.
\end{definition}

\begin{lemma}\label{lem:Hausfact}
    Let $Y$ be a Hausdorff space, $f:X\to Y$ be a continuous map and $h:X\to \mathrm{Hau}(X)$ be the projection. Then there exists a unique continuous map $g:\mathrm{Hau}(X)\to Y$ such that $f=g\circ h$.
\end{lemma}

\begin{definition}\label{def:charvariety}
    The $\SL_2(F)$-character variety $\mathcal{X}_F(\Gamma)$ is defined as the Hausdorffization of the quotient space $\mR_F(\Gamma)/\SL_2(F)$:
    \[\mathcal{X}_F(\Gamma)=\mathrm{Hau}(\mR_F(\Gamma)/\SL_2(F)).\]
    The two topologies on $\mR_F(\Gamma)$ induce quotient topologies on $\mathcal{X}_F(\Gamma)$, called the Zariski topology and the non-Archimedean topology of $\mathcal{X}_F(\Gamma)$, respectively. 
    
    The subspace of $\mathcal{X}_F(\Gamma)$ consisting of equivalence classes of bounded (resp. unbounded) representations, is denoted by $\mathcal{X}^b_F(\Gamma)$ (resp. $\mathcal{X}^{ub}_F(\Gamma)$). 
\end{definition}
\begin{remark}
    It does not immediately follow from the definition of Hausdorffization that $\mathcal{X}^b_F(\Gamma)$ and $\mathcal{X}^{ub}_F(\Gamma)$ are disjoint. However, from Proposition \ref{prop_openness} and the continuity of $\Psi$ proved in the proof of Theorem \ref{thm:continuity}, it follows that $\mathcal{X}^b_F(\Gamma)$ and $\mathcal{X}^{ub}_F(\Gamma)$ are disjoint open sets in the non-Archimedean topology.
\end{remark}

\begin{lemma}\label{lem:charHitchin}
    Let $Q:\mR_{F}(\Gamma)\to \mathcal{X}_{F}(\Gamma)$ be the quotient map. If $\mathcal{X}_{F}(\Gamma)$ is equipped with the non-Archimedean topology, then there exists a unique continuous map \begin{align}\label{formular:charHitchin}
        \Phi:\mathcal{X}_{F}(\Gamma)\to H^0(\Sigma,\mathcal{K}_{\Sigma}^{\otimes 2}),
    \end{align}
    such that $\overline{\Phi}=\Phi\circ Q$. 
\end{lemma}
\begin{proof}
    If two representations $\rho,\rho'\in\mR_F(\Gamma)$ are conjugate, then $\ell_{\rho}\equiv \ell_{\rho'}$. Consequently, $\overline{\Phi}$ descends to the quotient space $\mR_F(\Gamma)/\SL_2(F)$ to be a continuous map. And since $H^0(\Sigma,\mathcal{K}_{\Sigma}^{\otimes 2})$ is Hausdorff, the universal property of Hausdorffization (Lemma~\ref{lem:Hausfact}) implies that $\overline{\Phi}$ factors uniquely through $Q$. Thus, there exists a unique continuous map $\Phi:\mathcal{X}_{F}(\Gamma)\to H^0(\Sigma,\mathcal{K}_{\Sigma}^{\otimes 2})$ such that $\overline{\Phi}=\Phi\circ Q$.
\end{proof}

\begin{definition}
    The continuous map $\Phi$ given in Lemma \ref{lem:charHitchin} is defined to be the \emph{non-Archimedean Hitchin map}. 
\end{definition}

\begin{remark}
    With this definition, Theorem~\ref{thm:continuity} holds by construction. The main content of the continuity result lies in Theorem~\ref{thm:preconti}, which establishes the continuity of the lift $\overline{\Phi}$ on the representation variety.
\end{remark}

Finally, we give a more concrete description of $\mathcal{X}_F^{ub}(\Gamma)$. The proof of the following Proposition follows closely the argument in \cite[Section 2.3]{invitationCullerShalen}.

\begin{proposition}
    Within the space $\mR^{ub}_F(\Gamma)$ of unbounded representations, the closure equivalence relation $\sim_c$ is an equivalence relation. 
\end{proposition}
\begin{proof}
    By Proposition \ref{prop_trace_determines_ubirr}, if $\rho$ is an unbounded irreducible representation, then the orbit $[\rho]$ of $\rho$ coincides with the set $\{\rho'\in\mR_F(\Gamma):\trace_{\rho'}\equiv\trace_{\rho}\}$. By Lemma~\ref{lem_tracepolynomial}, this set is Zariski closed; hence $\overline{[\rho]}=[\rho]$.

    Now suppose $\rho$ is reducible. One can repeat the calculation in the proof of \cite[Lemma 14]{invitationCullerShalen}---noting that algebraic closedness is not used there---to see that there exists a diagonal representation $\rho_D:\Gamma\to F^{\times}\oplus F^{\times}$ lying in the Zariski closure $\overline{[\rho]}$. Moreover, the Zariski closure $\overline{[\rho]}$ of $[\rho]$ contains exactly one orbit $[\rho_D]$ of a completely reducible representation. 

    Suppose $\overline{[\rho_1]}\cap \overline{[\rho_2]}\ne \varnothing$ and $\overline{[\rho_2]}\cap \overline{[\rho_3]}\ne \varnothing$, for $\rho_1,\rho_2,\rho_3$ unbounded. If $\rho_1$ is irreducible, then $[\rho_1]=[\rho_2]=[\rho_3]$. If $\rho_1$ is reducible, then $\overline{[\rho_1]}$, $\overline{[\rho_2]}$ and $\overline{[\rho_3]}$ intersect at the same orbit of some completely reducible representation.

    This proves the transitivity of $\sim_c$. The reflexivity and symmetry follow form the definition.
\end{proof}

\begin{theorem}
    $\mR^{ub}_F(\Gamma)/\sim_c=\mathcal{X}_F^{ub}(\Gamma)$.
\end{theorem}
\begin{proof}
    By \cite[Corollary 4.3.4]{noteonchar}, it suffices to show that $\mR^{ub}_F(\Gamma)/\sim_c$ is Hausdorff. 

    Now consider the following map \[\Theta:\mR_F(\Gamma)\to F^{\Gamma}:\rho\mapsto \trace_{\rho},\] where $F^{\Gamma}$ is the space of $F$-valued functions on $\Gamma$, with weak topology induced by the non-Archimedean norm on $F$. As the proof of the continuity of $\Psi$ in Theorem \ref{thm:continuity}, we can show that $\Theta$ is continuous when $\mR_F(\Gamma)$ is equipped with the non-Archimedean topology.

    As in the proof of \cite[Lemma 14]{invitationCullerShalen}, one can show that, for two completely reducible representations $\rho_D,\rho'_D$, we have $[\rho_D]=[\rho'_D]$ if and only if $\trace_{\rho_D}\equiv \trace_{\rho'_D}$. Moreover, for a reducible representation $\rho$, if $\overline{[\rho]}$ contains a completely reducible $\rho_D$, then $\trace_{\rho}\equiv\trace_{\rho_D}$. Thus $\Theta$ is constant on the Zariski closure of any orbit in $\mR^{ub}_F(\Gamma)$.  

    It follows that $\Theta$ descends to an \emph{injective} continuous map \begin{align}\label{formula_tracemap}\theta:\mR^{ub}_F(\Gamma)/\sim_c\to F^{\Gamma}:\overline{[\rho]}\mapsto \trace_{\rho}.\end{align}
    Therefore, $\mR^{ub}_F(\Gamma)/\sim_c$ must be Hausdorff.
\end{proof}

\begin{remark}
    Indeed, the trace map \eqref{formula_tracemap} can be reduced to the injection \begin{align}\label{formula:truetracemap}t:\mathcal{X}^{ub}_F(\Gamma)\to F^{2^n-1}:\overline{[\rho]}\mapsto \Big(\trace\big(\rho(\gamma_{j_1}\cdots\gamma_{j_m})\big)\Big)_{1\le j_1<\cdots<j_m\le n,\,m\le n}.\end{align}
    Furthermore, if $F$ is algebraically closed, this injection can be extended to the entire character variety $\mathcal{X}_{F}(\Gamma)$ and endows $\mathcal{X}_{F}(\Gamma)$ with the structure of an algebraic variety. See \cite{CullerShalen83,invitationCullerShalen}.
\end{remark}

\section{Image of the Hitchin map}
\noindent
In this section, we prove that the image of the non-Archimedean Hitchin map 
lies in the space of Jenkins–Strebel differentials. 
We recall basic properties of quadratic differentials and show that the associated 
leaf space is a $\mathbb{Z}$-tree of infinite valence.

\begin{comment}
In this section, we show that the image of the non-Archimedean Hitchin map \[\Phi:\mathcal{X}_{F}(\Gamma)\to H^0(\Sigma,\mathcal{K}_{\Sigma}^{\otimes 2})\] is contained in the space of Jenkins-Strebel differential. We begin by recalling some basic results on the trajectory structure of quadratic differentials, and introducing the definition of Jenkins-Strebel differentials. Second, we show that the Hopf differential associated with an equivariant harmonic map into a $\ZZ$-tree with vertices (i.e., not isomorphic to $\RR$) is always a Jenkins-Strebel differential. Since the Bruhat-Tits $T_F$ is such a $\ZZ$-tree, this will allow us to prove Theorem \ref{thm:imageJS}. 

Finally, we study the leaf space (on the universal covering) of a Jenkins-Strebel differential, showing that it must be a $\ZZ$-tree of infinite valence. 
\end{comment}

\subsection{Trajectory structure of quadratic differentials}

In this subsection, we discuss the trajectory structure of quadratic differentials and introduce the notion of Jenkins–Strebel differentials.

\begin{definition}
    Consider a quadratic differential $q$ on a (not necessarily compact) Riemann surface $\Sigma$. A vertical arc $\gamma$ of a quadratic differential $q$ is a curve $\gamma\subset \Sigma$ that satisfies $|\langle\Re\sqrt{q},\dot{\gamma}\rangle|\equiv 0$. A vertical trajectory or simply a \textbf{trajectory} of $q$ is a maximal vertical arc. A trajectory that is periodic is called a closed trajectory.
\end{definition}

\begin{definition}
	\label{def_JSdifferential}
	A non-zero quadratic differential $q$ on a compact Riemann surface $\Sigma$ is called a \textbf{Jenkins–Strebel differential} if its non-closed trajectories cover a set of measure zero.
\end{definition}

\begin{definition}\label{def_limitset}
    Let $\gamma$ be a non-closed trajectory of $q$. We can parameterize $\gamma:(-\infty,+\infty)\to\Sigma$ such that $\gamma(t_1)\ne\gamma(t_2)$ for all $t_1\ne t_2$. If $\gamma(t)$ converges to a single point as $t$ tends to $+\infty$ or $-\infty$, then $\gamma$ is called a \textbf{critical trajectory}, and the limit must be a zero of $q$. The union of critical trajectories as well as zeros of $q$, denoted by $\MC$, is called the \textbf{critical graph} of $q$.
\end{definition}

\begin{theorem}\label{thm_noemptyinterior}
    Consider a nonzero quadratic differential $q$ on a compact Riemann surface $\Sigma$. Suppose $\gamma$ is a non-closed trajectory of $q$ whose limit set contains more than two points. Then the closure of $\gamma$ has nonempty interior.
\end{theorem}
\begin{proof}
    This result follows directly from Theorem 11.1 and 11.2 in \cite{strebel84quadratic}.
\end{proof}

\begin{proposition}\textup{\cite[Theorem 20.1]{strebel84quadratic}}\label{prop_JenkinsStrebel}
Consider a nonzero quadratic differential $q$ on a compact Riemann surface $\Sigma$. The following are equivalent:
\begin{itemize}
    \item[\emph{(i)}] $q$ is a Jenkins-Strebel differential;
    \item[\emph{(ii)}] Every non-critical trajectory of $q$ is closed;
    \item[\emph{(iii)}] The critical graph of $q$ is compact.
\end{itemize}
\end{proposition}

The set $\Sigma\setminus \MC$ consists of a collection of cylinders ${C_k}$ swept out by closed trajectories. We call ${C_k}$ the \emph{characteristic cylinders} of the Jenkins–Strebel differential $q$, and the isotopy class of the closed trajectories in a characteristic cylinder $C_k$ the \emph{core curve} of $C_k$. The supremum of the transverse distances of two closed trajectories in $C_k$ is called the \emph{height} of $C_k$. See Figure \ref{fig_JSgenus2} for an example of a Jenkins-Strebel differential.

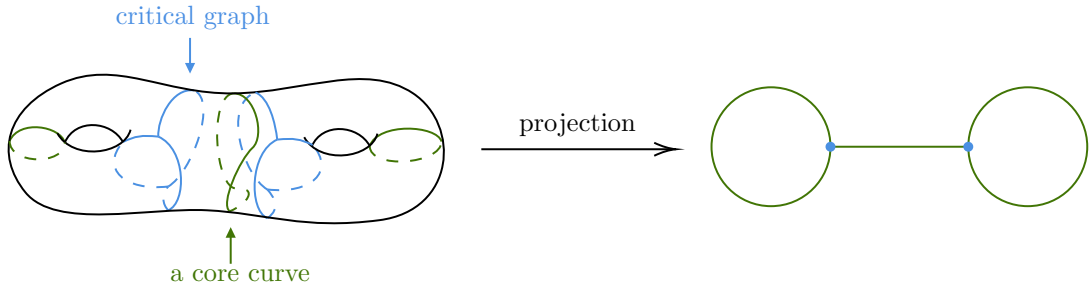
\begin{figure}[!h]
    \centering

\tikzset{every picture/.style={line width=0.75pt}} %set default line width to 0.75pt        

\begin{tikzpicture}[x=0.75pt,y=0.75pt,yscale=-1,xscale=1]

%Curve Lines [id:da6433760908601954] 
\draw [color={rgb, 255:red, 65; green, 117; blue, 5 }  ,draw opacity=1 ] [dash pattern={on 4.5pt off 4.5pt}]  (232.13,139.21) .. controls (232.13,151.21) and (268,148.67) .. (268.34,137) ;
%Curve Lines [id:da3420690686636929] 
\draw [color={rgb, 255:red, 65; green, 117; blue, 5 }  ,draw opacity=1 ]   (232.13,139.21) .. controls (236.5,127.17) and (266,128.67) .. (268.34,137) ;
%Curve Lines [id:da5621823504573927] 
\draw [color={rgb, 255:red, 74; green, 144; blue, 226 }  ,draw opacity=1 ] [dash pattern={on 4.5pt off 4.5pt}]  (180.47,161.17) .. controls (168.47,154.17) and (157.5,118.67) .. (174.5,112.17) ;
%Curve Lines [id:da3567288891257713] 
\draw [color={rgb, 255:red, 65; green, 117; blue, 5 }  ,draw opacity=1 ] [dash pattern={on 4.5pt off 4.5pt}]  (163.5,112.67) .. controls (150.5,113.67) and (154,161.17) .. (165,160.67) .. controls (176,160.17) and (170.5,168.67) .. (161.5,172.17) ;
%Curve Lines [id:da9719806874821317] 
\draw [color={rgb, 255:red, 65; green, 117; blue, 5 }  ,draw opacity=1 ]   (163.5,112.67) .. controls (174.5,112.67) and (179.5,134.17) .. (173,140.17) .. controls (166.5,146.17) and (156.5,168.67) .. (161.5,172.17) ;
%Curve Lines [id:da6593800277650189] 
\draw [color={rgb, 255:red, 65; green, 117; blue, 5 }  ,draw opacity=1 ]   (50.5,136.67) .. controls (55.5,125.67) and (76,129.17) .. (78.34,137.5) ;
%Curve Lines [id:da34900917046029245] 
\draw [color={rgb, 255:red, 74; green, 144; blue, 226 }  ,draw opacity=1 ]   (185.47,136) .. controls (171.47,142.17) and (171,171.67) .. (180.5,175.67) ;
%Curve Lines [id:da8545666780865765] 
\draw [color={rgb, 255:red, 74; green, 144; blue, 226 }  ,draw opacity=1 ] [dash pattern={on 4.5pt off 4.5pt}]  (180.5,175.67) .. controls (187.5,170.67) and (183.97,165.17) .. (180.47,161.17) ;
%Curve Lines [id:da24618439849670493] 
\draw [color={rgb, 255:red, 74; green, 144; blue, 226 }  ,draw opacity=1 ] [dash pattern={on 4.5pt off 4.5pt}]  (180.47,161.17) .. controls (202.47,163.67) and (214.97,152.17) .. (203.33,138.71) ;
%Curve Lines [id:da5309856587447546] 
\draw [color={rgb, 255:red, 74; green, 144; blue, 226 }  ,draw opacity=1 ]   (143.5,111.17) .. controls (132.5,109.67) and (125,125.67) .. (125.5,134.5) ;
%Curve Lines [id:da4354494381511913] 
\draw [color={rgb, 255:red, 74; green, 144; blue, 226 }  ,draw opacity=1 ]   (125.5,134.5) .. controls (139.5,140.67) and (140,167.67) .. (130.5,171.67) ;
%Curve Lines [id:da20945434360613335] 
\draw [color={rgb, 255:red, 74; green, 144; blue, 226 }  ,draw opacity=1 ]   (125.5,134.5) .. controls (117,133.67) and (114,134.17) .. (107.63,137.21) ;
%Curve Lines [id:da6067458964602852] 
\draw [color={rgb, 255:red, 74; green, 144; blue, 226 }  ,draw opacity=1 ] [dash pattern={on 4.5pt off 4.5pt}]  (130.5,159.67) .. controls (142.5,152.67) and (155.5,116.17) .. (143.5,111.17) ;
%Curve Lines [id:da21208149384030417] 
\draw [color={rgb, 255:red, 74; green, 144; blue, 226 }  ,draw opacity=1 ] [dash pattern={on 4.5pt off 4.5pt}]  (130.5,171.67) .. controls (127.5,168.67) and (127,163.67) .. (130.5,159.67) ;
%Curve Lines [id:da4526951912174453] 
\draw [color={rgb, 255:red, 74; green, 144; blue, 226 }  ,draw opacity=1 ] [dash pattern={on 4.5pt off 4.5pt}]  (130.5,159.67) .. controls (108.5,162.17) and (96,150.67) .. (107.63,137.21) ;
%Shape: Polygon Curved [id:ds3029037308790814] 
\draw   (163.5,112.67) .. controls (199,112.17) and (230,96.67) .. (254,112.67) .. controls (278,128.67) and (275.5,171.67) .. (237,176.67) .. controls (198.5,181.67) and (180,173.17) .. (150,171.67) .. controls (120,170.17) and (74,181.67) .. (58.5,162.67) .. controls (43,143.67) and (48.5,120.17) .. (76.5,107.67) .. controls (104.5,95.17) and (128,113.17) .. (163.5,112.67) -- cycle ;
%Curve Lines [id:da996075059102906] 
\draw    (74.58,132.78) .. controls (82.31,146.8) and (108.22,143.72) .. (111.5,131) ;
%Curve Lines [id:da4960023067671172] 
\draw    (78.34,137.5) .. controls (83.7,124.99) and (101.83,126.57) .. (107.63,137.21) ;
%Curve Lines [id:da17476538482209913] 
\draw    (199.08,134.78) .. controls (206.81,148.8) and (232.72,145.72) .. (236,133) ;
%Curve Lines [id:da6683101179174458] 
\draw    (202.84,139.5) .. controls (208.2,126.99) and (226.33,128.57) .. (232.13,139.21) ;
%Curve Lines [id:da33669657809271947] 
\draw [color={rgb, 255:red, 74; green, 144; blue, 226 }  ,draw opacity=1 ]   (174.5,112.17) .. controls (181.5,115.17) and (184.5,127.67) .. (185.47,136) ;
%Curve Lines [id:da5409966134899635] 
\draw [color={rgb, 255:red, 74; green, 144; blue, 226 }  ,draw opacity=1 ]   (185.47,136) .. controls (193.97,135.17) and (196.97,135.67) .. (203.33,138.71) ;
%Curve Lines [id:da06033014875724474] 
\draw [color={rgb, 255:red, 65; green, 117; blue, 5 }  ,draw opacity=1 ] [dash pattern={on 4.5pt off 4.5pt}]  (50.5,136.67) .. controls (50.5,148.67) and (78,149.17) .. (78.34,137.5) ;
%Straight Lines [id:da37395564631730815] 
\draw [color={rgb, 255:red, 65; green, 117; blue, 5 }  ,draw opacity=1 ]   (162,198.67) -- (162,183.17) ;
\draw [shift={(162,180.17)}, rotate = 90] [fill={rgb, 255:red, 65; green, 117; blue, 5 }  ,fill opacity=1 ][line width=0.08]  [draw opacity=0] (5.36,-2.57) -- (0,0) -- (5.36,2.57) -- cycle    ;
%Straight Lines [id:da44390245766478587] 
\draw [color={rgb, 255:red, 74; green, 144; blue, 226 }  ,draw opacity=1 ]   (141.5,84.67) -- (141.5,99.67) ;
\draw [shift={(141.5,102.67)}, rotate = 270] [fill={rgb, 255:red, 74; green, 144; blue, 226 }  ,fill opacity=1 ][line width=0.08]  [draw opacity=0] (5.36,-2.57) -- (0,0) -- (5.36,2.57) -- cycle    ;
%Shape: Circle [id:dp3281405395364717] 
\draw  [color={rgb, 255:red, 65; green, 117; blue, 5 }  ,draw opacity=1 ] (404.5,139.6) .. controls (404.5,123.03) and (417.93,109.6) .. (434.5,109.6) .. controls (451.07,109.6) and (464.5,123.03) .. (464.5,139.6) .. controls (464.5,156.17) and (451.07,169.6) .. (434.5,169.6) .. controls (417.93,169.6) and (404.5,156.17) .. (404.5,139.6) -- cycle ;
%Shape: Circle [id:dp42335064281650847] 
\draw  [color={rgb, 255:red, 65; green, 117; blue, 5 }  ,draw opacity=1 ] (534,139.6) .. controls (534,123.03) and (547.43,109.6) .. (564,109.6) .. controls (580.57,109.6) and (594,123.03) .. (594,139.6) .. controls (594,156.17) and (580.57,169.6) .. (564,169.6) .. controls (547.43,169.6) and (534,156.17) .. (534,139.6) -- cycle ;
%Straight Lines [id:da6775816446833556] 
\draw [color={rgb, 255:red, 65; green, 117; blue, 5 }  ,draw opacity=1 ]   (464.5,139.6) -- (534,139.6) ;
%Shape: Circle [id:dp9675263478012066] 
\draw  [color={rgb, 255:red, 74; green, 144; blue, 226 }  ,draw opacity=1 ][fill={rgb, 255:red, 74; green, 144; blue, 226 }  ,fill opacity=1 ] (462.5,139.6) .. controls (462.5,138.5) and (463.4,137.6) .. (464.5,137.6) .. controls (465.6,137.6) and (466.5,138.5) .. (466.5,139.6) .. controls (466.5,140.7) and (465.6,141.6) .. (464.5,141.6) .. controls (463.4,141.6) and (462.5,140.7) .. (462.5,139.6) -- cycle ;
%Shape: Circle [id:dp5774362552069016] 
\draw  [color={rgb, 255:red, 74; green, 144; blue, 226 }  ,draw opacity=1 ][fill={rgb, 255:red, 74; green, 144; blue, 226 }  ,fill opacity=1 ] (532,139.6) .. controls (532,138.5) and (532.9,137.6) .. (534,137.6) .. controls (535.1,137.6) and (536,138.5) .. (536,139.6) .. controls (536,140.7) and (535.1,141.6) .. (534,141.6) .. controls (532.9,141.6) and (532,140.7) .. (532,139.6) -- cycle ;
%Straight Lines [id:da8472619723886795] 
\draw    (288.4,140.87) -- (384.53,140.87) ;
\draw [shift={(386.53,140.87)}, rotate = 180] [color={rgb, 255:red, 0; green, 0; blue, 0 }  ][line width=0.75]    (10.93,-3.29) .. controls (6.95,-1.4) and (3.31,-0.3) .. (0,0) .. controls (3.31,0.3) and (6.95,1.4) .. (10.93,3.29)   ;

% Text Node
\draw (130,200) node [anchor=north west][inner sep=0.75pt]   [align=left] {\textcolor[rgb]{0.25,0.46,0.02}{{\small a core curve}}};
% Text Node
\draw (102.5,67) node [anchor=north west][inner sep=0.75pt]   [align=left] {\textcolor[rgb]{0.29,0.56,0.89}{{\small critical graph}}};
% Text Node
\draw (306,121.27) node [anchor=north west][inner sep=0.75pt]  [font=\small] [align=left] {projection};

\end{tikzpicture}

    \caption{A Jenkins-Strebel differential on a genus-$2$ surface and its leaf space}
    \label{fig_JSgenus2}
\end{figure}

\subsection{Hopf differentials associated with $\ZZ$-trees}

\begin{lemma}\label{lem_hausdorff}
    Suppose $(T,d)$ is a $\ZZ$-tree whose vertex set is nonempty. If $\Gamma$ is a group acting on $T$ isometrically, then the quotient space $T/\Gamma$ is a Hausdorff space.
\end{lemma}
\begin{proof}
    It suffices to verify that for any point $x\in T$, the $\Gamma$-orbit $[x]=\Gamma.x=\{g.x:g\in\Gamma\}$ is discrete. 

    Let $x$ be a vertex point of $T$, then $[x]$ is a set of vertices, since the $\Gamma$-action is isometric. And by assumption on $T$, $[x]$ must be a discrete subset of $T$.

    Let $x$ be an edge point of $T$. Suppose, for contradiction, there exists a sequence $\{g_i\}$ in $\Gamma$, such that $g_i.x\ne x$ converges to $x$.
    
    Since the vertex set of $T$ is a nonempty closed subset, there exists a vertex $z$ closest to $x$. Let $\gamma:[0,d(x,z)]\to T$ be the unique geodesic segment between $x$ and $z$, with $\gamma(0)=x$ and $\gamma(d(x,z))=z$. Then by definition of $z$, there are no other vertices of $T$ on $\gamma$. And $\forall g\in\Gamma$, $g.\gamma$ is the geodesic segment from the edge point $g.x$ to the vertex $g.z$.

    We can choose a sufficiently small $\epsilon>0$, and fix an orientation on $B_{\epsilon}(x)$ such that the metric ball is identified with the open interval $(-\epsilon,\epsilon)$. Since $g_i.x \to x$, for sufficiently large $i$, we have $g_i.x \in B_\epsilon(x)$, and the geodesic $g_i.\gamma$ intersects $B_\epsilon(x)$ nontrivially. The orientation of $g_i.\gamma$ may agree or disagree with the orientation of $B_\epsilon(x)$. By passing to subsequence and changing the orientation of $B_{\epsilon}(x)$, we may always assume the orientations of $g_i.\gamma$ and $B_{\epsilon}(x)$ are always coincide on their intersections.
     
    Note that for any $g_i$, there are no vertices on the interior of $g_i.\gamma$. Therefore, for $m$ and $n$ sufficiently large, $g_m.\gamma\cup g_n.\gamma$ is isometric with a closed interval.

\begin{figure}[!h]
\centering

\tikzset{every picture/.style={line width=0.75pt}} %set default line width to 0.75pt        

\begin{tikzpicture}[x=0.75pt,y=0.75pt,yscale=-1,xscale=1]

%Straight Lines [id:da08636087795911951] 
\draw    (360,90.4) -- (582.67,90.07) ;
%Flowchart: Summing Junction [id:dp060276420991423385] 
\draw  [color={rgb, 255:red, 208; green, 2; blue, 27 }  ,draw opacity=1 ][fill={rgb, 255:red, 208; green, 2; blue, 27 }  ,fill opacity=1 ] (444.2,89.93) .. controls (444.2,88.18) and (442.78,86.77) .. (441.03,86.77) .. controls (439.28,86.77) and (437.87,88.18) .. (437.87,89.93) .. controls (437.87,91.68) and (439.28,93.1) .. (441.03,93.1) .. controls (442.78,93.1) and (444.2,91.68) .. (444.2,89.93) -- cycle ; \draw  [color={rgb, 255:red, 208; green, 2; blue, 27 }  ,draw opacity=1 ] (443.27,87.69) -- (438.79,92.17) ; \draw  [color={rgb, 255:red, 208; green, 2; blue, 27 }  ,draw opacity=1 ] (438.79,87.69) -- (443.27,92.17) ;
%Flowchart: Summing Junction [id:dp5467289176369267] 
\draw  [fill={rgb, 255:red, 0; green, 0; blue, 0 }  ,fill opacity=1 ] (422.8,90.13) .. controls (422.8,88.38) and (421.38,86.97) .. (419.63,86.97) .. controls (417.88,86.97) and (416.47,88.38) .. (416.47,90.13) .. controls (416.47,91.88) and (417.88,93.3) .. (419.63,93.3) .. controls (421.38,93.3) and (422.8,91.88) .. (422.8,90.13) -- cycle ; \draw   (421.87,87.89) -- (417.39,92.37) ; \draw   (417.39,87.89) -- (421.87,92.37) ;
%Flowchart: Summing Junction [id:dp45217382505097825] 
\draw  [color={rgb, 255:red, 74; green, 144; blue, 226 }  ,draw opacity=1 ][fill={rgb, 255:red, 74; green, 144; blue, 226 }  ,fill opacity=1 ] (392.3,90.7) .. controls (392.3,88.95) and (390.88,87.53) .. (389.13,87.53) .. controls (387.38,87.53) and (385.97,88.95) .. (385.97,90.7) .. controls (385.97,92.45) and (387.38,93.87) .. (389.13,93.87) .. controls (390.88,93.87) and (392.3,92.45) .. (392.3,90.7) -- cycle ; \draw  [color={rgb, 255:red, 74; green, 144; blue, 226 }  ,draw opacity=1 ] (391.37,88.46) -- (386.89,92.94) ; \draw  [color={rgb, 255:red, 74; green, 144; blue, 226 }  ,draw opacity=1 ] (386.89,88.46) -- (391.37,92.94) ;
%Flowchart: Summing Junction [id:dp37976843906965696] 
\draw  [color={rgb, 255:red, 74; green, 144; blue, 226 }  ,draw opacity=1 ][fill={rgb, 255:red, 74; green, 144; blue, 226 }  ,fill opacity=1 ] (513.97,90.2) .. controls (513.97,88.45) and (512.55,87.03) .. (510.8,87.03) .. controls (509.05,87.03) and (507.63,88.45) .. (507.63,90.2) .. controls (507.63,91.95) and (509.05,93.37) .. (510.8,93.37) .. controls (512.55,93.37) and (513.97,91.95) .. (513.97,90.2) -- cycle ; \draw  [color={rgb, 255:red, 74; green, 144; blue, 226 }  ,draw opacity=1 ] (513.04,87.96) -- (508.56,92.44) ; \draw  [color={rgb, 255:red, 74; green, 144; blue, 226 }  ,draw opacity=1 ] (508.56,87.96) -- (513.04,92.44) ;
%Flowchart: Summing Junction [id:dp7624697834784093] 
\draw  [color={rgb, 255:red, 208; green, 2; blue, 27 }  ,draw opacity=1 ][fill={rgb, 255:red, 208; green, 2; blue, 27 }  ,fill opacity=1 ] (563.47,89.83) .. controls (563.47,88.08) and (562.05,86.67) .. (560.3,86.67) .. controls (558.55,86.67) and (557.13,88.08) .. (557.13,89.83) .. controls (557.13,91.58) and (558.55,93) .. (560.3,93) .. controls (562.05,93) and (563.47,91.58) .. (563.47,89.83) -- cycle ; \draw  [color={rgb, 255:red, 208; green, 2; blue, 27 }  ,draw opacity=1 ] (562.54,87.59) -- (558.06,92.07) ; \draw  [color={rgb, 255:red, 208; green, 2; blue, 27 }  ,draw opacity=1 ] (558.06,87.59) -- (562.54,92.07) ;
%Shape: Arc [id:dp7603341791194295] 
\draw  [draw opacity=0] (373.94,100.16) .. controls (371.89,97.56) and (370.67,94.27) .. (370.67,90.7) .. controls (370.67,87.05) and (371.94,83.7) .. (374.08,81.07) -- (385.97,90.7) -- cycle ; \draw   (373.94,100.16) .. controls (371.89,97.56) and (370.67,94.27) .. (370.67,90.7) .. controls (370.67,87.05) and (371.94,83.7) .. (374.08,81.07) ;  
%Shape: Arc [id:dp16075946795889817] 
\draw  [draw opacity=0] (466.79,98.96) .. controls (468.84,96.36) and (470.07,93.07) .. (470.07,89.5) .. controls (470.07,85.85) and (468.79,82.5) .. (466.66,79.87) -- (454.77,89.5) -- cycle ; \draw   (466.79,98.96) .. controls (468.84,96.36) and (470.07,93.07) .. (470.07,89.5) .. controls (470.07,85.85) and (468.79,82.5) .. (466.66,79.87) ;  

% Text Node
\draw (376.8,68.6) node [anchor=north west][inner sep=0.75pt]  [font=\footnotesize]  {$\textcolor[rgb]{0.29,0.56,0.89}{g}\textcolor[rgb]{0.29,0.56,0.89}{_{m}}\textcolor[rgb]{0.29,0.56,0.89}{.x}$};
% Text Node
\draw (431.5,69.8) node [anchor=north west][inner sep=0.75pt]  [font=\footnotesize]  {$\textcolor[rgb]{0.82,0.01,0.11}{g}\textcolor[rgb]{0.82,0.01,0.11}{_{n}}\textcolor[rgb]{0.82,0.01,0.11}{.x}$};
% Text Node
\draw (415.3,94.9) node [anchor=north west][inner sep=0.75pt]  [font=\footnotesize]  {$x$};
% Text Node
\draw (491.9,69.6) node [anchor=north west][inner sep=0.75pt]  [font=\footnotesize]  {$\textcolor[rgb]{0.29,0.56,0.89}{g}\textcolor[rgb]{0.29,0.56,0.89}{_{m}}\textcolor[rgb]{0.29,0.56,0.89}{.z}$};
% Text Node
\draw (545.6,69.8) node [anchor=north west][inner sep=0.75pt]  [font=\footnotesize]  {$\textcolor[rgb]{0.82,0.01,0.11}{g}\textcolor[rgb]{0.82,0.01,0.11}{_{n}}\textcolor[rgb]{0.82,0.01,0.11}{.z}$};
% Text Node
\draw (364.27,107.6) node [anchor=north west][inner sep=0.75pt]  [font=\footnotesize]  {${\textstyle B_{\epsilon }( x)}$};

\end{tikzpicture}

         \caption{Overlaping geodesics}
         \label{overlap}
     \end{figure}
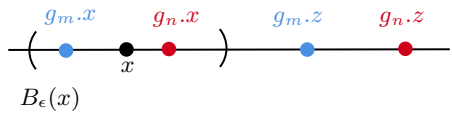

    Now we have $d(g_m.z,g_n.z)=d(g_m.x,g_n.x)$. As $m,n\to\infty$, $d(g_m.x,g_n.x)$ tends to zero, and hence the set of vertices $\{g_i.z\}$ has a limit point. This contradicts the assumption that the vertex set of $T$ is discrete.
\end{proof}

\begin{theorem}\label{thm_HopfdiffisJS}
    Consider a nonconstant $\Gamma$-equivariant harmonic map $u:\tSigma\to T$, where $T$ is a $\ZZ$-tree with nonempty vertex set. Then the Hopf differential $q_u$ is a Jenkins-Strebel differential on $\Sigma$. 
\end{theorem}
\begin{proof}
    We shown in Lemma \ref{lem_hausdorff} that $\mathcal{G}:=T/\Gamma$ is a Hausdorff space. The $\Gamma$-equivariant map $u$ reduce to a continuous map $\underline{u}:\Sigma \to T/\Gamma$.

    Note that when restricted to a trajectory of $p^\ast q_u$, $u$ is constant. Thus for any trajectory $\beta$ of $q_u$ on $\Sigma$, $\underline{u}|_{\beta}$ is constant. By Theorem \ref{thm_noemptyinterior}, if $\beta$ is non-critical and non-closed, then there is a small disk $D$ on $\Sigma$, such that $\beta\cap D$ is dense in $D$.

    Since $T/\Gamma$ is Hausdorff, $\underline{u}|_{\beta\cap D}$ extends to $D$ as a constant map. It follows that on the preimage $p^{-1}(D)$, the map $u:\tSigma\to T$ is constant. Thus $p^\ast q_u$ vanishes on $p^{-1}(D)$. However, since $q_u$ is holomorphic and nonzero, this is impossible. Therefore, any non-critical trajectory of $q_u$ is closed, and hence by Proposition \ref{prop_JenkinsStrebel}, $q_u$ is Jenkins-Strebel.
\end{proof}

Theorem \ref{thm:imageJS} is a corollary of Theorem \ref{thm_HopfdiffisJS}, since $T_F$ is always a $\ZZ$-tree with nonempty vertex set.

\begin{remark}
   The fact that a trajectory dense in somewhere forces a harmonic map to be constant was previously used in \cite{Wolf_JSdiffs}.
\end{remark}

\subsection{Leaf space of a Jenkins-Strebel differential}
In this subsection, we discuss the leaf space structure of a Jenkins-Strebel differential. As an application, we show that the leaf space $T_q$ on $\tSigma$ has infinite valence.  We will also show that a quadratic differential $q$ is Jenkins-Strebel if and only if $T_q$ is a $\ZZ$-tree. 

Consider a nonzero Jenkins-Strebel differential $q$ on $\Sigma$. Denote the critical graph of $q$ by $\MC$. Let $\widetilde{\MC}:=p^{-1}(\MC)$, where $p:\tSigma\to \Sigma$ is the universal covering map. We write $\MC=\MC_1\sqcup\cdots\sqcup \MC_m$, where each $\MC_i$ is a connected component of $\MC$. 
Let  
\[
H_i:=\mathrm{Im}((\iota_i)_*:\pi_1(\MC_i)\to \Gamma)
\]  
be the image of the map induced by the inclusion $\iota_i:\MC_i\hookrightarrow\Sigma$. Then  
\[
\sharp\widetilde{\MC}_i=[\pi_1(\MC_i):H_i].
\]  

\begin{lemma}
    The inclusion $(\iota_i)_*:\pi_1(\MC_i)\to \Gamma$ is injective for any $1\le i\le m$.
\end{lemma}
\begin{proof}
    It suffices to show that a connected component of $\widetilde{\MC}_i=:p^{-1}(\MC_i)$, say $\widetilde{\MC}_{ij}$, is a simplicial tree of finite valence. In this case, $p|_{\widetilde{\MC}_{ij}}$ is the universal covering of $\MC_i$, and hence $H_i\cong \pi_1(\MC_i)$.

    Suppose, for contradiction, that $\widetilde{\MC}_{ij}$ is not simply connected, then there is a simple loop $\gamma$ contained in $\widetilde{\MC}_{ij}$. Since $\tSigma\approx \mathbb{H}^2$, $\gamma$ bounds a disk $D$ on $\tSigma$. Then the singular foliation $\tilde{\F}_q$ on $\tSigma$, restricts to a singular foliation $\mathcal{F}_D:=\mathcal{F}_q|_D$ on $D$. However, the boundary $\partial D$ lies in a leaf of $\mathcal{F}_D$, while $\mathcal{F}_D$ has no critical point of nonnegative index. This contradicts with the Euler-Poincar\'e formula for a disk.
\end{proof}

As each component of the graph contains the zeros of $q$, $\MC_i$ contains a vertex and each vertex of $\MC_i$ has valence at least three, hence $\pi_1(\MC_i)$ must be a free group of rank at least two. Consequently, $H_i$ is a nonabelian free subgroup of $\Gamma$.  
As every finite-index subgroup of $\Gamma$ is a surface group, $H_i$ cannot be of finite index.  
Thus $\sharp\widetilde{\MC}_i$ and $\sharp\widetilde{\MC}$ are countably infinite.

Recall the construction of the $\RR$-tree $T_q$ associated to a quadratic differential $q$, given after Definition \ref{def:Gammatree}. Now we have:

\begin{proposition}\label{prop_TqisZtree}
Let $q$ be a Jenkins-Strebel differential on $\Sigma$. Then $T_q$ is a $\ZZ$-tree with countably infinitely many vertices, and each vertex has countably infinitely many incident edges.
\end{proposition}
\begin{proof}
    Let $\pi:\tSigma\to T_q$ be the projection. We show that the image of $\widetilde{\MC}$ under $\pi$ is exactly the vertex set of $T_q$. Note that any $x\in T_q\setminus \pi(\widetilde{\MC})$ is an edge point, so it suffices to show that each point $x\in \pi(\widetilde{\MC})$ is a vertex of infinite valence.

    Consider a small neighborhood $N(\MC_i)$ of $\MC_i\subset\MC$ on $\Sigma$, whose boundary consists of finitely many closed trajectories of $q$. $N(\MC_i)$ has a deformation retraction onto $\MC_i$. Let $\beta\subset \partial N(\MC_i)$ be such a closed trajectory. $N(\MC_i)$ can be lifted to an open subset $N(\widetilde{\MC}_i)$ of $\tSigma$ with countably infinite connected components, and each its component (denoted by $N(\widetilde{\MC}_{ij})$) is a neighborhood of a component $\widetilde{\MC}_{ij}$ of $\widetilde{\MC}_i$. $\tilde{\beta}=:p^{-1}(\beta)$ is contained in the boundary of $N(\widetilde{\MC}_i)$. Each component of $\tilde{\beta}$ is projected to an edge point of $T_q$. 
    
    We next show that $\pi(\tilde{\beta}\cap\partial N(\widetilde{\MC}_{ij}))$ is a countably infinite set of edge points. 

    First, we show that $\tilde{\beta}\cap\partial N(\widetilde{\MC}_{ij})$ has countably infinite connected components. Recall that $H_i=(\iota_i)_{\ast}\pi_1(\MC_i)\cong\pi_1(\MC_i)$ is a free subgroup of rank at least two. And $\beta$ represents a nontrivial element in $H_i$. 
    
    On the universal covering $\tSigma$, the stabilizer of $\widetilde{\MC}_{ij}$ (as well as $N(\widetilde{\MC}_{ij})$) in the fundamental group $\Gamma$, is conjugate equivalent to $H_i$. And similarly, the stabilizer of a component of $\tilde{\beta}$ is conjugate equivalent to the subgroup $\langle[\beta]\rangle$ with one generator. Suppose $\tilde{\beta}_{ik}$ is a component of $\tilde{\beta}\cap\partial N(\widetilde{\MC}_{ij})$, then from the discussion above, we deduce that its stabilizer is a subgroup of infinite index of the stabilizer of $\widetilde{\MC}_{ij}$, since $\beta$ represents an element of $H_i$. Therefore, under the action of the stabilizer of $\widetilde{\MC}_{ij}$, $\tilde{\beta}_{ik}$ is mapped to infinitely many components in $\tilde{\beta}\cap\partial N(\widetilde{\MC}_{ij})$.

    Next, we verify that each two distinct components $\tilde{\beta}_{ik}$ and $\tilde{\beta}_{il}$ in $\tilde{\beta}\cap\partial N(\widetilde{\MC}_{ij})$ have different image under the projection $\pi:\tSigma\to T_q$. It suffices to show that \[d(\pi(\tilde{\beta}_{ik}),\pi(\tilde{\beta}_{il}))=\inf_{\gamma}\int_{\gamma}\langle\Re\sqrt{p^\ast q},\dot{\gamma}\rangle dt>0,\]where $\gamma$ ranges over all piece-wise $C^1$ arcs connecting $\tilde{\beta}_{ik}$ and $\tilde{\beta}_{il}$. That is, to show that the transverse distance between $\tilde{\beta}_{ik}$ and $\tilde{\beta}_{il}$ is nonzero.

    Given an arc $\gamma$ connecting $\tilde{\beta}_{ik}$ and $\tilde{\beta}_{il}$, consider the push-forward $\underline{\gamma}:=p(\gamma)$ (w.l.o.g. assumed to be a loop). We have \[\int_{\gamma}\langle\Re\sqrt{p^\ast q},\dot{\gamma}\rangle dt=\int_{\underline{\gamma}}\langle\Re\sqrt{q},\dot{\underline{\gamma}}\rangle dt.\]

    It follows that $\underline{\gamma}$ intersects $\MC$ nontrivially. Indeed, if $\underline{\gamma}\cap \MC=\varnothing$, then $\underline{\gamma}$ is contained in a cylinder and hence $[\underline{\gamma}]=n[\beta]$ for some $n\in\ZZ$, which implies $\tilde{\beta}_k=\tilde{\beta}_l$, contradicting to the assumption. However, when $\underline{\gamma}\cap \MC\ne\varnothing$, $\int_{\underline{\gamma}}\langle\Re\sqrt{q},\dot{\underline{\gamma}}\rangle dt$ has a positive lower bound, which is given by the transverse distance between $\beta$ and the critical graph $\MC$. Therefore, each vertex of $T_q$ has infinite valence.

    Finally, we show that each component of $\widetilde{\MC}$ maps to a distinct vertex of $T_q$, and hence $T_q$ has infinitely many vertices. Suppose $\widetilde{\MC}_{ij}$ and $\widetilde{\MC}_{kl}$ are lifts of different components $\MC_i$ and $\MC_k$ of $\MC$, under the covering map $p \colon \tSigma \to \Sigma$. Then the distance between $\pi(\widetilde{\MC}_{ij})$ and $\pi(\widetilde{\MC}_{kl})$ is bounded below by the transverse distance between $\MC_i$ and $\MC_k$.

    If $\widetilde{\MC}_{ij}$ and $\widetilde{\MC}_{il}$ are two distinct components of $p^{-1}(\MC_i)$, then for any arc $\gamma$ connecting them, the projection $\underline{\gamma}$ in $\Sigma$ represents an element in $\Gamma \setminus \pi_1(\MC_i)$. There exists a characteristic cylinder $C_k$ of $q$ such that $\partial C_k \cap \MC_i \ne \varnothing$, and $\underline{\gamma}$ intersects the core curve of $C_k$ nontrivially. Suppose not. Then $\underline{\gamma}$ is homotopic to a loop in $\MC_i$, hence $[\underline{\gamma}] \in \pi_1(\MC_i)$, which is a contradiction. Now if $\underline{\gamma}$ has nonempty intersection with the core curve of $C_k$, then the transverse length of $\underline{\gamma}$ is bounded below by the height of $C_k$. Thus $\pi(\widetilde{\MC}_{ij})$ and $\pi(\widetilde{\MC}_{il})$ are different vertices in $T_q$.
\end{proof}

\begin{proposition}
    A quadratic differential $q$ is Jenkins-Strebel if and only if $T_{q}$ is a $\mathbb{Z}$-tree. 
\end{proposition}
\begin{proof}
Suppose $T_q$ is a $\ZZ$-tree. By Theorem \ref{thm_smallaction} and \cite[Lemma 2.1]{farbwolf01}, $T_q$ must have a vertex point. Moreover, by \cite[Proposition 2.2]{daskalopoulos2000morganshalen}, the projection $\pi:\tSigma\to T_q$ is a harmonic map whose Hopf differential is precisely $q$. Thus Theorem \ref{thm_HopfdiffisJS} implies that $q$ is Jenkins-Strebel. Conversely, when $q$ is Jenkins-Strebel, we have shown in Proposition \ref{prop_TqisZtree} that $T_q$ is a $\ZZ$-tree.
\end{proof}

\section{Density and Non-Small Action}
In this section, we investigate Zariski and topological density of representations 
and their geometric consequences for the $\Gamma$-action on $T_F$.
\subsection{Zariski and topological density}

For $H=\rho(\Gamma)$, let $\overline{H}^{\mathrm{Zar}}$ denote the Zariski closure of $H$ with respect to the Zariski topology of $\mathrm{SL}_2(F)$, and let $\overline{H}$ denote the closure of $H$ in the analytic topology of $\mathrm{SL}_2(F)$, which is induced by the non-Archimedean topology on $F$.

\begin{definition}
We say that $\rho$ is \emph{Zariski dense} if $\overline{H}^{\mathrm{Zar}} = \SL_2(F)$, and \emph{topologically dense} if $\overline{H} = \SL_2(F)$.
\end{definition}

\begin{proposition}\label{prop_Zariskidensity}
    An unbounded representation $\rho$ is Zariski dense if and only if it is irreducible.
\end{proposition}
\begin{proof}
    If $\rho$ is not Zariski dense, then $\overline{H}^{\mathrm{Zar}}$ is a proper algebraic subgroup of $\SL_2(F)$. Every proper algebraic subgroup of $\SL_2(F)$ is virtually solvable, hence contains no nonabelian free subgroups. As a result, for any two elements $\gamma_1,\gamma_2\in\Gamma$, the subgroup $\langle\rho(\gamma_1),\rho(\gamma_2)\rangle<\SL_2(F)$ cannot be isometric with to the free group $F_2$. By Lemma \ref{lem_freegrpinnonabelian}, $\rho$ must be reducible.

    Conversely, if $\rho$ is reducible, then $\rho(\Gamma)$ lies in a Borel subgroup of $\SL_2(F)$, wihch is a proper algebraic subgroup of $\SL_2(F)$. Thus $\rho$ is not Zariski dense.
\end{proof}

\begin{corollary}
    The space of Zariski dense representations is open in $\mR^{ub}_F(\Gamma)$.
\end{corollary}
\begin{proof}
    It suffices to show that the set of unbounded irreducible representations are open.
    
    For such a representation $\rho$, we can find an element $\gamma$ in the commutator subgroup $[\Gamma,\Gamma]$ such that $\ell_{\rho}(\gamma)> 0$. Then we can choose a sufficiently small neighborhood $\mathcal{U}$ of $\rho$ so that $\forall\rho'\in\mathcal{U}$, $\ell_{\rho'}(\gamma)>0$. Therefore, $\mathcal{U}$ consists of unbounded irreducible representations.
\end{proof}

\begin{lemma}\textup{\cite[Chapter II.1.4, Theorem 2]{serre80trees}}\label{lem_topdensetransitive}
    If $\rho$ is topologically dense, then the fundamental domain of the corresponding $\Gamma$-action on $T_F$ is an edge (of length $1$).
\end{lemma}

Here the \emph{fundamental domain} refers to a connected subgraph 
\( D \subset T_F \) such that $$
T_F \;=\; \bigcup_{\gamma \in \Gamma} \gamma \cdot D 
\quad \text{and} \quad 
\gamma \cdot D^\circ \cap D^\circ = \emptyset 
\quad \text{for all } \gamma \neq e .$$ Topological density implies that the action on the set of edges of $T_F$ is transitive.

\begin{corollary}
    If $\rho:\Gamma\to \SL_2(F)$ is a topologically dense representation, then $\ell_{\rho}$ is nonzero and nonabelian. In particular, topological density implies unbounded and Zariski density.
\end{corollary}
\begin{proof}
    By Lemma \ref{lem_topdensetransitive}, $\Gamma$ acts transitively on the set of edges of $T_F$. Thus $\Gamma$ has neither a global fixed point nor a fixed end, hence $\ell_{\rho}$ is nonzero and nonabelian. 
\end{proof}

The converse is false in general: a representation may be Zariski dense without being topologically dense. Indeed, there exist representations whose images are discrete free subgroups of $\SL_2(F)$ of rank at least 2 \cite{Schottkygroups,Schottkygrouplocalfield}.
By Proposition~\ref{prop_Zariskidensity}, such subgroups are Zariski dense.

\subsection{Small actions}
A $\Gamma$-action on a $\mbR$-tree is called \emph{small} if edge stablizers do not contain rank two free groups. It is called minimal if it has no proper $\Gamma$-subtree. The work of Morgan-Otal \cite{morganotal93}, Skora \cite{Skora96} reflects the importance of minimal small action, see also Farb-Wolf \cite{farbwolf01}.
\begin{theorem}[\cite{morganotal93,Skora96,farbwolf01}]\label{thm_smallaction}
Let $\Sigma$ be a closed Riemann surface of genus $\geq 2$ with fundamental group $\Gamma$. 
\begin{itemize}
    \item[\emph{(i)}] If $\Gamma$ acts isometrically on an $\mathbb{R}$-tree $T$ in a small and minimal way, 
    then there exists a holomorphic quadratic differential $q$ on $\Sigma$ such that 
    $T$ is $\Gamma$-equivariantly isometric to the dual tree of the vertical measured foliation of $q$ lifted to $\tSigma$.
    \item[\emph{(ii)}] Conversely, for any holomorphic quadratic differential $q$ on $\Sigma$, 
    the dual tree of its vertical measured foliation carries a small isometric $\Gamma$-action.
\end{itemize}
\end{theorem}

Consider a $\Gamma$-tree $T$. Let $u:\tSigma\to T$ be a $\Gamma$-equivariant harmonic map and $q_u$ be the associated Hopf differential. Then there exists a folding map $f:T_{q_u}\to T$. 

\begin{theorem}[\cite{morganotal93,Skora96,farbwolf01}]
	Suppose the $\Gamma$-action on $T$ is small and minimal, then the folding map $f:T_q\to T$ is an isomorphism.
\end{theorem}

 Consider a representation $\rho:\Gamma\to\SL_2(F)$. Let $e$ be an edge of the Bruhat-Tits tree $T_F$, we define the edge stablizer:
\[
(\Gamma)_e:=\{\gamma\in\Gamma:\ \rho(\gamma)\in \mathrm{Stab}_{\mathrm{SL}_2(F)}(e)\}.
\]
\begin{proposition}
\label{prop:nonsmall}
Suppose $F$ is a \emph{locally compact} non-Archimedean field. Let $\rho:\Gamma\to \SL_2(F)$ be an unbounded representation. The induced $\Gamma$–action on $T_F$ is not small. In particular, there exists an edge $e$ of $T_F$
such that for every $\gamma. e$ with $\gamma\in\Gamma$, the edge stablizer $(\Gamma)_{\gamma. e}$ contains a free group of rank two.
\end{proposition}
\begin{proof}
    Recall that if $F$ is locally compact, $T_F$ is a simplicial tree of finite valence. 
    
    If $\rho$ is small, then the folding $T_q\to T_{\min}$ is an isomorphism, where $T_{\min}$ is the minimal $\Gamma$-subtree in $T_F$. However, by Proposition \ref{prop_TqisZtree}, $T_q$ is of infinite valence, which cannot be a subtree of $T_F$. 
    
    Therefore, the $\Gamma$-action cannot be small and there exists an edge $e$ such that $(\Gamma)_e$ contains a free group of rank two. Moreover, for $e'=\gamma. e$, we have $(\Gamma)_{e'}=\gamma\,(\Gamma)_e\,\gamma^{-1}$; hence non-smallness holds for every edge in $\Gamma. e$.
\end{proof}

Suppose $\rho$ is topological dense, then $T_{\min}=T_F$. As topological dense representations acts transitively on the edge of the Bruhat-Tits building, we conclude
\begin{corollary}
Suppose $\rho$ is topological dense, then every edge stablizer contains a rank two free group.
\end{corollary}

\bibliography{references}
\bibliographystyle{plain}

\Addresses
\end{document}

%% file: convention.tex
\newcommand{\MC}{\mathcal{C}}
\newcommand{\mR}{\mathcal{R}}

\newcommand{\tSigma}{\widetilde{\Sigma}}

\newcommand{\SO}{\mathrm{SO}}

\newcommand{\mbR}{\mathbb{R}}

\newcommand{\RR}{\mathbb{R}}

\newcommand{\ZZ}{\mathbb{Z}}

\newcommand{\SL}{\mathrm{SL}}
\newcommand{\trace}{\mathrm{tr}}

\newcommand{\F}{\mathcal{F}}